\documentclass[11pt, a4paper, UKenglish]{article}
\usepackage[UKenglish]{babel}
\usepackage{Test_1}
\usepackage[maxbibnames=99,style=numeric,sortcites=true]{biblatex}
\usepackage{hyperref}
\title{Homogenization of the Navier-Stokes equations in a randomly perforated domain in the inviscid limit}
\author{Richard M. H\"ofer$^1$\footnote{\texttt{richard.hoefer@ur.de}}\and Eleni Hübner-Rosenau$^1$\footnote{\texttt{eleni.huebner-rosenau@ur.de}}}
\date{$^1$ Faculty of Mathematics, University of Regensburg, Germany\\[2ex]
\today}

\begin{document}
\maketitle
	\begin{abstract}
	    \noindent We study the behaviour of the solution $u_{\eps}$ to the Navier-Stokes equations with vanishing viscosity and a non-slip condition in a randomly perforated domain. We consider the space $\R^3$ where we remove $N$ holes that are  i.i.d. distributed. The behaviour depends on the particle size $\eps^\alpha=N^{-\alpha/3}$ and the viscosity $\eps^\gamma=N^{-\gamma/3}$ of the fluid. We prove quantitative convergence results to a function $u$, provided that the local Reynolds number is small, in the subcritical ($\alpha+\gamma>3$) and critical ($\alpha+\gamma=3$) regime. In the first case, $u$ solves the Euler equations, whereas in the second case $u$ solves the Euler-Brinkman equations.
        This extends the results of \url{https://doi.org/10.1088/1361-6544/acfe56} from the periodic to the random setting. We only treat the case $\alpha>2$ so that the particles do not overlap with overwhelming probability.
	\end{abstract}
    \section{Introduction} Mixtures of fluids with many small particles, called \textit{suspensions}, appear in a lot in nature and industrial applications -- be it in the form of blood, dust storms or air-fuel injections in engines. Hence, this topic has attracted considerable interest over the last decades.
    
    The most accurate but also more difficult way to describe these suspensions is to let the particles evolve according to Newton's law and prescribe the behaviour of the surrounding medium at the particles. However, this results in a large number of variables and equations to solve, so this is rather impractical. Hence, we try to look for the so-called \textit{effective behaviour} -- easier equations describing the approximate behaviour of the macroscopic mixture, taking into account the overall effect of the many small particles.
    
    Since the full system of the fluid with moving particles and their interaction is mathematically quite challenging, there will usually be some simplifications to be made. In particular, the model we study here considers only stationary particles with a no-slip boundary condition (that is, the fluid does not move at the particles). Heuristically, there are usually three cases, depending on the size of the particles -- the first results go back to the experiments of Darcy \cite{Darcy.56} and the mathematically rigorous works of Marchenko and Khruslov \cite{MarKhr64}, Tartar \cite{Tartar.80} and Allaire \cite{Allaire.90a,Allaire.90b}:
    \begin{enumerate}
        \item \emph{subcritical case:} the particles are very small and do not influence the behaviour of the surrounding medium much, so the equations stay the same 
        \item \emph{critical case:} the particles are critically sized and lead to an extra fiction term (as proposed by Brinkman in \cite{Brinkman.49}),
        \item \emph{supercritical case:} the particles are larger and the additional friction dominates the behaviour of the mixture, resulting in  Darcy's law.
    \end{enumerate}
    Further mathematical results on this topic concern the Stokes and Navier-Stokes equations with incompressible (\cite{FeiNamNec.16,LuYang.23,Pan.25}) or compressible (\cite{Masmoudi.02,Oschmann.22,HoeNecOsc.25}) fluids. 
    
    Most of the results have considered a fixed positive Reynolds number. However, fluids with low viscosity (or rather large Reynolds numbers) are physically interesting, such as, for example, in the modelling of sprays (\cite{BarDes.06,BalMikWhe.10}). There are only few mathematical results dealing with the combined homogenization and inviscid limit. In \cite{Mikelic91, Mikelic95, MARUSICPALOKA200097}, the problem is studied in a periodic geometry where the particle size is of the same order as the inter-particle distance, resulting in Darcy-type equations. Results dealing with smaller particles are \cite{lacave2015}, where the subcritical $2$-dimensional regime is studied, and \cite{Richard.23}, where all corresponding three regimes in the $3$-dimensional case are analyzed leading to the Euler, the Euler-Brinkman and Darcy's law in the subcritical, critical and supercritical case, respectively. The results \cite{lacave2015, Richard.23} are still restricted to periodic perforations. In this work, we extend the results in \cite{Richard.23} for the subcritical and critical case  to the randomly perforated domains.
    
    Considering the random setting is natural from the application oriented point of view and has been extensively studied for related homogenization problems in fluids with fixed viscosity, for instance in \cite{GiuHoe18,CarHil.20,Giunti.21,HoeJan.24,BelOsc.23}. In this article we will assume the particle centres to be i.i.d., which is closest to the setting of \cite{HoeJan.24}, though we use different methods and also prove convergence in expectation (not in probability like in \cite{HoeJan.24}). In contrast, the centres in \cite{GiuHoe18,Giunti.21,BelOsc.23} are marked Poisson point processes, whereas in \cite{CarHil.20} a more general framework of random points is considered. However, in all of these works only the Stokes equations are considered. In order to deal with potentially overlapping particles, we will not consider all ranges of particle sizes and viscosities studied in \cite{Richard.23}, but only where the particles have a size much smaller than $N^{-2/3}$, and only the critical and subcritical regimes.

	\subsection{Setting and Notation}
	Let $\mathcal{T}\Subset B_{1/4}(0)$, the reference particle, be a closed smooth set such that $B_1(0)\setminus \mathcal{T}$ is connected and $0\in \mathring{\mathcal{T}}$. We then perforate the entire space $\R^3$ with $N$ randomly distributed holes, rescaled to size $N^{-\alpha/3}$. To be precise, we define the perforated domain as
	\begin{align*}
		\Omega_\eps=\R^3\setminus \bigcup_{i=1}^N \mathcal{T}_i^\eps, \quad\qquad\mathcal{T}_i^\eps=x_i^\eps+\eps^\alpha \mathcal{T}.
	\end{align*}
	Here, $\alpha>1$ is a parameter and $\eps:=N^{-1/3}$. The centres $x_i^\eps$ are $N$ random particle centres with common density
    \begin{align}
        \rho\in L^\infty(\R^3)\cap \mathcal P(\R^3)
    \end{align}
    which has compact support. That is to say, to generate the particle centres, we consider the space $(\R^3)^N$ with the probability measure $\rho^{\otimes N}$ (along with the standard Borel $\sigma$-algebra).
    
    We will study the behaviour of $u_\eps$, the solution to the incompressible Navier-Stokes equations with viscosity $\eps^\gamma$ (where $\gamma>0$), i.e.
	\begin{equation}\label{eq: pde for u_eps}
			\begin{aligned}
			\del_t u_\eps+(u_\eps\cdot \nabla)u_\eps-\mu_0\eps^\gamma\Delta u_\eps+\nabla p_\eps&=f_\eps &&\text{ in }(0,T)\times\Omega_\eps\\
			\divg(u_\eps)&=0&&\text{ in }(0,T)\times\Omega_\eps\\
			u_\eps|_{\partial\Omega_\eps}&=0&&\text{ on }(0,T)\\
			u_\eps(0,\cdot)&=u_{0,\eps}&&\text{ in }\Omega_\eps,
		\end{aligned}
	\end{equation}
	where $f_\eps\in L^2(0,T;L^2(\R^3;\R^3))$ and $u_{0,\eps}\in L^2_{\sigma}(\Omega_\eps;\R^3)$ are given source terms and initial data respectively. Here, $L^2_{\sigma}(\Omega_\eps;\R^3)$ denotes the space
    \begin{align}
        L^2_{\sigma}(\Omega_\eps;\R^3)=\{v\in L^2(\Omega_\eps;\R^3)\:|\:\divg(v)=0, v\cdot n=0\text{ on }\partial\Omega_\eps\}
    \end{align}
    where the condition $\divg(v)=0$ has to be understood in the distributional sense.
	It is known (see \cite{Sohr_2014} or \cite{Robinson_Rodrigo_Sadowski_2016} for an overview of this and other classical results) that a Leray solution exists, i.e. a weak solution satisfying the energy estimate (for any $t\in[0,T]$)
	\begin{align}\label{eq:energy ineq}
		\frac{1}{2}\norm{u_\eps(t)}^2_{L^2(\Omega_\eps)}+\mu_0\eps^\gamma\norm{\nabla u_\eps}^2_{L^2((0,t)\times\Omega_\eps)}\leq 	\frac{1}{2}\norm{u_{\eps,0}}^2_{L^2(\Omega_\eps)}+\int_0^t\int_{\Omega_\eps}f_\eps\cdot u_\eps\dx\ds.
	\end{align} 
	We will show that when $\alpha+\gamma=3$ (critical scaling) with $\alpha\in(2,3)$ and assuming that $f_\eps, u_{0,\eps}$ converge to suitably smooth limit functions $f,u_0$, then $u_\eps$ converges to a function $u$ solving the PDE
	\begin{equation}\label{eq: pde for u}
		\begin{aligned}
					\del_t u+(u\cdot\nabla)u+\nabla p+\rho\mu_0\mathcal{R}u&=f\quad\text{in }(0,T)\times\R^3\\
			\divg(u)&=0\quad\text{in }(0,T)\times\R^3\\
			u(0,\cdot)&=u_0\quad\text{in }\R^3,
		\end{aligned}
	\end{equation}
	and we can also estimate the rate of convergence. Here, $\mathcal{R}$ is the so-called resistance matrix that can be determined by solving the Stokes equations with suitable boundary conditions on $\R^3\setminus\mathcal{T}$. This extra linear term $\mu_0\rho\mathcal{R}u$, which represents the average effect that all the particles have on the fluid, is also called the \textit{Brinkman force}.
    
    In the subcritical case ($\alpha+\gamma>3$) with $\alpha>2$, we show that $u_\eps$ converges to a solution of the Euler equations (corresponding to $\mathcal{R}=0$):
    	\begin{equation}\label{eq: pde for u subcrit}
		\begin{aligned}
					\del_t u+(u\cdot\nabla)u+\nabla p&=f\quad\text{in }(0,T)\times\R^3\\
			\divg(u)&=0\quad\text{in }(0,T)\times\R^3\\
			u(0,\cdot)&=u_0\quad\text{in }\R^3,
		\end{aligned}
	\end{equation}
	Now we state the main results, which involve quantitative convergence estimates, uniform in time. Note that we assume some regularity for the limit function $u$, which exists at least for small times.
	\begin{thm}\label{thm:critical case}
		Let $\rho\in L^\infty(\R^3)\cap\mathcal{P}(\R^3)$ be a density with compact support. Moreover, let $\alpha\in(2,3)$ and $\gamma=3-\alpha$ and $\mu_0=1$. Let $u_0\in H^4(\R^3;\R^3)$, $f\in C(0,T;H^2(\R^3;\R^3))$ and for $T>0$ let $(u,p)\in C^1(0,T;H^4(\R^3;\R^3))\times L^\infty(0,T;H^3_{\text{loc}}(\R^3))$ be a solution to the Euler-Brinkman equations \eqref{eq: pde for u}. Moreover, for $0<\eps<1$, let $u_{0,\eps}\in L^2_\sigma(\Omega_\eps;\R^3)$, $f_\eps\in L^2(0,T;L^2(\Omega_\eps;\R^3))$ and let $u_\eps\in L^2(0,T; H^1_0(\Omega_\eps;\R^3))\cap C(0,T;L^2(\Omega_\eps;\R^3))$ be a Leray solution to the Navier-Stokes equations \eqref{eq: pde for u_eps}.\\
		Then, for any $0<\la<1/6$ there exists a constant $C>0$, which depends on $\la$, the reference particle $\mathcal{T}$, monotonously on $T$, $\norm{f}_{L^\infty(0,T;H^2(\R^3))}$, $\norm{u}_{C^1(0,T;H^3(\R^3))}$ and $\norm{\nabla p}_{L^\infty(0,T;H^2(\R^3))}$, $\norm{\rho}_{L^\infty(\R^3)}$, $\operatorname{diam}(\supp(\rho))$, such that for all $0\leq t\leq T$
		\begin{equation}\label{eq:final estimate critical}
			\begin{aligned}
				\mathbb{E}\left[\norm{(u-u_\eps)(t)}_{L^2(\Omega_\eps)}^2\right]&\leqc \mathbb{E}\left[\norm{f_\eps-f}^2_{L^2((0,T)\times\R^3)}\right]+(\eps^{6-2\alpha}+\eps^{3(\alpha-2)}+\eps^{\alpha-3/2}),
			\end{aligned}
		\end{equation}
		where $\rho_\eps:=\frac{1}{N}\sum_{i=1}^N\delta_{x_i^\eps}$ is the empirical density.
	\end{thm}
    \begin{thm}\label{thm:subcritical case}
       Let $\rho\in L^\infty(\R^3)\cap\mathcal{P}(\R^3)$ be a density with compact support. Moreover, let $\alpha>2$ and $3-\alpha<\gamma\leq \alpha$ and $\mu_0>0$. Let $u_0\in H^4(\R^3;\R^3)$, $f\in C(0,T;H^2(\R^3;\R^3))$ and for $T>0$ let $(u,p)\in C^1(0,T;H^4(\R^3;\R^3))\times L^\infty(0,T;H^3_{\text{loc}}(\R^3))$ be a solution to the Euler equations \eqref{eq: pde for u subcrit}. Moreover, for $0<\eps<1$, let $u_{0,\eps}\in L^2_\sigma(\Omega_\eps;\R^3)$, $f_\eps\in L^2(0,T;L^2(\Omega_\eps;\R^3))$ and let $u_\eps\in L^2(0,T; H^1_0(\Omega_\eps;\R^3))\cap C(0,T;L^2(\Omega_\eps;\R^3))$ be a Leray solution to the Navier-Stokes equations \eqref{eq: pde for u_eps}.\\
		Then, for any $0<\la<1/6$ there exist constants $M>0$ and $C>0$ that depends on $\la$, the reference particle $\mathcal{T}$, monotonously on $T$, $\norm{f}_{L^\infty(0,T;H^2(\R^3))}$, $\norm{u}_{C^1(0,T;H^3(\R^3))}$ and $\norm{\nabla p}_{L^\infty(0,T;H^2(\R^3))}$, $\norm{\rho}_{L^\infty(\R^3)}$, $\operatorname{diam}(\supp(\rho))$, such that if either $\alpha>\gamma$ or $\mu_0\geq M$, we have for all $0\leq t\leq T$
				\begin{align*}
			\mathbb{E}\left[\norm{(u-u_\eps)(t)}_{L^2(\Omega_\eps)}^2\right]\leqc &\mathbb{E}\left[\norm{f_\eps-f}^2_{L^2((0,T)\times\R^3)}\right]+(\eps^{2\gamma}+\eps^{3(\alpha-2)}+\eps^{2\alpha+\gamma-9/2}),
		\end{align*}
		where $\rho_\eps=\frac{1}{N}\sum_{i=1}^N\delta_{x_i^\eps}$ is the empirical density.
    \end{thm}
    \subsection{Informal derivation}
    We will now briefly discuss an informal derivation of the result. For simplicity, we assume that the hole $\mathcal{T}$ is spherical, so that $\mathcal{T}=B_{r}(0)=:B$. The argument works the same for general shapes, with some small changes.
    
    In order to better understand the behaviour of $u_\eps$, we take a closer look at an individual hole. Namely, we define
    \begin{align*}
    v_\eps^i(x,t)&=u_\eps(x_i^\eps+\eps^{\alpha}x,t),\\
    q_\eps^i(x,t)&=\eps^{\alpha-\gamma} p_\eps(x_i^\eps+\eps^\alpha x,t).
    \end{align*}
    Then, locally around the hole centered at $x_i^\eps$, we can study the behaviour of $v_\eps^i$: $v_\eps^i$ solves the PDE
    \begin{equation}
        \begin{aligned}
            \eps^{2\alpha-\gamma}\partial_t v_\eps^i+\eps^{\alpha-\gamma}(v_\eps^i\cdot\nabla)v_\eps^i-\Delta v_\eps^i+\nabla q_\eps^i&=\eps^{2\alpha-\gamma}f_\eps\quad\text{in }(0,T)\times (B_{\eps^{1-\alpha}}(0)\setminus B),\\
            \divg(v_\eps^i)&=0\quad\text{in }(0,T)\times (B_{\eps^{1-\alpha}}(0)\setminus B),\\
            v_\eps^i|_{\partial B}&=0\quad\text{in }(0,T),\\
            v_\eps^i|_{\partial B_{\eps^{1-\alpha}}}&=u_\eps(x_i^\eps+\eps^{\alpha}\cdot,t)|_{\partial B_{\eps^{1-\alpha}}(0)}\quad\text{in }(0,T).
        \end{aligned}
    \end{equation}
    We observe that if the local Reynolds number $Re_{loc}=\eps^{\alpha-\gamma}$ is negligible -- that is, when $\alpha>\gamma$ -- then $v_\eps^i$ solves approximately the (stationary) Stokes equations with appropriate boundary conditions. Moreover, for $\eps$ very small, we can formally replace the large ball $B_{\eps^{1-\alpha}}(0)$ by the whole space $\R^3$.
    
    Finally, to approximate the boundary conditions at infinity, we approximate $u_\eps(x_i^\eps+\eps^{\alpha}\cdot,t)|_{\partial B_{\eps^{1-\alpha}}}$ by simply $u(x_i^\eps)$. This is justified by the following reasoning: if we assume that the limit $u$ exists and is suitably smooth, then it should not vary much over the length scale $\eps^\alpha$ that we are considering here.
    
    Hence, we can formally say that $v_\eps^i$ is approximated by the solution $\widetilde{v}^i$ of the following, easier PDE:
    \begin{equation}\label{eq:PDE for u_eps locally Stokea approx}
        \begin{aligned}
            -\Delta \widetilde{v}^i+\nabla \widetilde{q}^i&=0\quad\text{in }\R^3\setminus B,\\
            \divg(\widetilde{v}^i)&=0\quad\text{in }\R^3\setminus B,\\
            \widetilde{v}^i|_{\partial B}&=0,\\
            \widetilde{v}^i&\rightarrow V:=u(x_i^\eps)\quad\text{at }\infty
        \end{aligned}
    \end{equation}
    This kind of equation is well-studied (see for example \cite[Chapter 4.9]{Batchelor_2000}) -- in the case of a sphere, this is known as Stokes' law -- and we have that
    \begin{align}
         -\Delta \widetilde{v}^i+\nabla \widetilde{q}^i&=-\mathcal{R}V\frac{\mathcal{H}^2|_{\partial B}}{|\partial B|}\quad\text{in }\R^3.
    \end{align}
    Here $\mathcal{R}$ is the previously mentioned (symmetric and positive) resistance matrix that depends on the shape of the hole $\mathcal{T}$ (in this case of a sphere, it is simply $6\pi r I_3$).
    
    With this result, we rescale the holes to size $\eps^\alpha$ and recall that in the equation for $u_\eps$, the viscosity is $\eps^\gamma$, so that the $i$-th (spherical) hole contributes locally the extra force
    \begin{align}
        F_i=-\mathcal{R}u(x_i^\eps)\eps^{\alpha+\gamma}\frac{\mathcal{H}^2|_{\partial\mathcal{T}_i^\eps}}{|\partial\mathcal{T}_i^\eps|}.
    \end{align}
    Now, since the holes are on average far away from each other compared to their size, we can approximate the total force to be simply the sum of the individual forces:
        \begin{align}
        F_{total}\approx\sum_i F_i.
    \end{align}
    Then, $u_\eps\approx\widetilde{u}_\eps$ with 
    \begin{align}
        \partial_t \widetilde{u}_\eps+(\widetilde{u}_\eps\cdot\nabla)\widetilde{u}_\eps-\eps^\gamma\Delta\widetilde{u}_\eps+\nabla \widetilde{p}_\eps+\eps^{\alpha+\gamma}\sum_{i=1}^N\mathcal{R}u(x_i^\eps)\frac{\mathcal{H}^2|_{\partial\mathcal{T}_i^\eps}}{|\partial\mathcal{T}_i^\eps|}=f_\eps\quad\text{in }\R^3.
    \end{align}
    Then, approximating 
    \begin{align}
        \frac{\mathcal{H}^2|_{\partial\mathcal{T}_i^\eps}}{|\partial\mathcal{T}_i^\eps|}\approx \delta_{x_i^\eps}
    \end{align}
    and using that $N=\eps^{-3}$ and that the centers $x_i^\eps$ are i.i.d. with common density $\rho$ (so that $\frac{1}{N}\sum_i \delta_{x_i^\eps}\approx\rho$), we have
    \begin{align} \label{eq:Brinkman force measures approx informal}
        \eps^{\alpha+\gamma}\sum_i\mathcal{R}u(x_i^\eps)\frac{\mathcal{H}^2|_{\partial\mathcal{T}_i^\eps}}{|\partial\mathcal{T}_i^\eps|}\approx \eps^{\alpha+\gamma-3}\rho\mathcal{R}u.
    \end{align}
    From this we can easily see why the critical scaling is $\alpha+\gamma=3$: this is exactly the case when the extra force is expected to be of order one. If $\alpha+\gamma>3$, then the friction term is of order $\eps^{\alpha+\gamma-3}\rightarrow 0$ (and note that in any of these cases, the vanishing viscosity does not produce boundary layer effects since we consider the whole space $\R^3$). In the supercritical scaling ($\alpha+\gamma<3$), the extra force term dominates the other terms and so we get a different kind of PDE which will not be discussed here, but one can look at \cite{Richard.23} for the result in the case of periodic particles, which is a kind of Darcy's law.
    \subsection{Elements of the proof}
    As in \cite{Richard.23}, to prove the quantitative estimate for $u_\eps-u$ we use a relative energy argument, which is classical for example in the proof of weak-strong uniqueness of the Euler and Navier-Stokes equations and has been used before to prove convergence of the Navier-Stokes equation to the Euler equation in the inviscid limit (see the survey \cite{wiedemann2017weakstronguniquenessfluiddynamics} and references therein).
    
    However, unlike in the classical uses of the relative energy method, here the limit function $u$ is not an admissible test function for the PDE for $u_\eps$, since $u$ does not satisfy Dirichlet boundary conditions at $\partial\Omega_\eps$. Therefore, we approximate $u$ by some suitable function $\check{u}_\eps\in L^\infty(0,T;H^1(\Omega_\eps))$, so that $\norm{\check{u}_\eps-u}$ is small in $L^\infty(L^2)$. 
    
    To construct $\check{u}_\eps$, we will make use of a (matrix-valued) multiplicative corrector function $w^\eps$ that satisfies $w^\eps = 0$ inside the particles. This corrector function is a variant of the oscillating test function used by Allaire in \cite{Allaire.90a}. It is  built on the solution to the resistance problem (which is reminiscent of the PDE \eqref{eq:PDE for u_eps locally Stokea approx}):
    \begin{equation}\label{eq: resistance problem}
    	\begin{aligned}
    		-\Delta w+\nabla q&=0\quad\text{in }\R^3\setminus\mathcal{T},\\
    		\divg(w)&=0\quad\text{in }\R^3\setminus\mathcal{T},\\
    		w&=\id\quad\text{at }\partial\mathcal{T}.
    	\end{aligned}
    \end{equation}
    More precisely, we define $w^\eps = \sum_i w^{\eps}_i$, where $w^\eps_i$ is a suitably scaled, translated, and truncated version of $w$ such that $w^\eps_i = 0$ in $\mathcal T_i^\eps$ and the supports of $w^\eps_i$  are disjoint. 
    
    Since the function $w^\eps u$ is not divergence-free, we then define $\check{u}_\eps = w^\eps u - B_\eps (u)$ where $B_\eps$ is a suitable Bogovskii type operator. As in \cite{Richard.23} and unlike Allaire, we need to truncate on a sufficiently small lengthscale where the local Reynolds number is very small.
    The key difference compared to \cite{Richard.23} is having to deal with the possible clustering of the holes, which is why we need to  choose the truncating length scale $\eta_{i,\eps}$ dependent on the particle $i$ so that the truncation domains do not overlap. On average,  this lengthscale will still be the same as in \cite{Richard.23} because most particles are well separated. 
    
    This method of the oscillating test functions $w^\eps$ has also been used in \cite{BelOsc.23}, and a similar approach also in \cite{GiuHoe18,Giunti.21} (though more elaborate due to the random radii of the particles). The qualitative results in \cite{GiuHoe18,Giunti.21, BelOsc.23}  are based on a duality approach that allows to use $w^\eps$ on the level of the test functions.   In \cite{CarHil.20} a variant of this method is used. The authors obtain quantitative results through a clever dual characterization of the $L^2_{\mathrm{loc}}$ norm. In \cite{HoeJan.24}, a more explicit approximation of $u$ is constructed by considering the monopoles induced by $u$ and using the variational structure of the Stokes equations. Neither this variational characterization nor the method of the dual $L^2_{\mathrm{loc}}$ representation seems applicable to the case of the full Navier-Stokes equations.

    Moreover, the possible clustering of the holes makes estimating the term generating the Brinkman  term (something alike making the approximation \eqref{eq:Brinkman force measures approx informal} rigorous) harder.
    Indeed, for the relative energy method to work, we need to quantify smallness of $\bar \rho_N - \rho$ in $H^{-1}$ for a suitable smeared out version $\bar \rho_N$ of the empirical density $\rho_N= \sum_i \delta_{x_i^\eps}$. The replacement of $\rho_N$ to $\bar \rho_N$ is essentially thanks to the corrector $w^\eps$ that mitigates the forces on $\partial \mathcal T_i^\eps$ to forces on $\partial B_{\eta_{i,\eps}}(x_i^\eps)$.
    In fact, relying on Poincaré type estimates, we will be able to smear out the empirical density even further, and define
		\begin{align*}
			\overline{\rho}_\eps =\sum_{i=1}^N \frac{1}{N \eps^{3(1-\la)}}\textbf{1}_{\widetilde{Q_i^\eps}}\in\mathcal{P}(\R^3).
		\end{align*}
where 	\begin{align}\label{eq:q_i tilde def}
		\widetilde{Q_{i,\eps}}=\widetilde{Q_{i,\eps}}(x_i^\eps)=x_i^\eps+[-\eps^{1-\la}/2,\eps^{1-\la}/2].
	\end{align}
    for $\lambda >0$.
    
    To estimate $\|\rho - \bar \rho_N\|_{H^{-1}(\R^3)}$ we rely on the following Proposition:
    \begin{prop}[{\cite[Lemma 5.33]{Santambrogio.2015}}]\label{prop:W2 dist L^infty ineq}
			Let $\nu_0,\nu_1\in L^{\infty}(\R^d)\cap \mathcal{P}(\R^d)$. Then
			\begin{align*}
				\norm{\nu_0-\nu_1}_{H^{-1}(\R^d)}\leq \max\{\norm{\nu_0}_{L^\infty(\R^d)},\norm{\nu_1}_{L^\infty(\R^d)}\}^{1/2}W_2(\nu_0,\nu_1).
			\end{align*}
		\end{prop}
        Here, we use the Wasserstein metric on the space of probability measures $\mathcal{P}(\R^d)$, which is defined by
        \begin{align}
            W_2(\mu,\nu)^2&=\inf_{\gamma\in \Gamma(\mu,\nu)}\int_{\R^d\times\R^d}|x-y|^2\hspace{.1cm}\gamma(\text{d}x,\text{d}y),\\
            \Gamma(\mu,\nu)&=\{\gamma\in\mathcal{P}(\R^d\times\R^d):(\pi_1)_\#\gamma=\mu, (\pi_2)_\#\gamma=\nu\}.
        \end{align}
        where we denote by $\pi_1$ (respectively $\pi_2$) the projections from $\R^d\times\R^d$ on the first (respectively second) coordinate, and by $f_\#\gamma$ the pushforward measure of $\gamma$ by some function $f$.
        
    Proposition \ref{prop:W2 dist L^infty ineq} motivates the definition
	\begin{align}\label{assumption: enlarged cubes overlap bounded}
		\mathcal{B}_{\la}^N=\left\{x\in (\R^{3})^N:\:\|\bar \rho_\eps \|_{L^\infty(\R^3)} \leq 16\norm{\rho}_{L^\infty(\R^3)}\right\}, \tag{$\mathcal{B}$}
	\end{align}
We will rely on the known results that the set $\mathcal{B}_{\la}$ has overwhelming probability and that $W_2(\rho_N, \rho)$ is typically very small as $\eps \to 0$, see Section \ref{sec:prop results}.

In order to avoid dealing with overlapping particles, we consider the set
	\begin{align}\label{eq:A0}\tag{$\mathcal{A}$}
		\mathcal{A}_{L,\alpha}^N=\{x\in (\R^{3})^N:\: d_i^\eps:=\min_{j\neq i=1,\dots,N}|x_i-x_j|\geq 2L\eps^{\alpha}\quad \text{for all }i=1,\dots,N\}.
	\end{align}
It is well known that one needs to assume $\alpha >2$ such that $\mathcal{A}_{L,\alpha}^N$ has overwhelming probability, see also below in Section \ref{sec:prop results}. 

This kind of probabilistic methods have been used similarly in \cite[see Proposition 2.4]{CarHil.20}. In \cite{BelOsc.23,GiuHoe18,Giunti.21}, one needs a more refined argument since also the radius of the holes is random - the holes are then divided in ''good" and ''bad" holes, where the ''bad" can still be separated into a finite number of families of holes that are sufficiently spaced.

We remark that in \cite{CarHil.20}, instead of the relation between the $H^{-1}$-norm and the $W_2$ distance (which only holds for $L^\infty$ measures), it is used that the $W_1$ distance coincides with the dual Lipschitz norm, which provides estimates also for dual Hölder norms. This is sufficient in \cite{CarHil.20} because the test functions are bounded in $H^2$. We cannot proceed similarly here since we only have a priori bounds for $u_\eps$ in $H^1$.

\medskip
    
    In the following, we will first prove some estimates for the corrector function and the Bogovskii operator. Then, we prove some probabilistic results, which then allow us to prove the main results.
	\section{Corrector estimates}
	Since we will prove in Section \ref{sec:prop results} that the set $\mathcal{A}_{1,\alpha}^N$ has almost full measure, we will in the following only consider configurations in $\mathcal{A}_{1,\alpha}^N$ (so that the holes have a positive distance). We then construct the corrector the same way as in \cite{Richard.23}, with the difference that the parameter $\eta=\eta_\eps$ now depends on the particle $i$. To be precise, let 
	\begin{align*}
		\eps^\alpha\leq\eta_{i,\eps}\leq\min\{\eps,d_i^\eps\}
	\end{align*}
	to be chosen later, where $d_i^\eps$ is the minimal distance between the $i$-th centre and the other centres as defined in \eqref{eq:A0}.
	Around each hole centered on $x_i^\eps$ we have the cube with side lengths $\eps$, that is, $Q_i^\eps=x_i^\eps+[-\eps/2, \eps/2]^3$. We split this cube in four parts, namely
	\begin{align*}
		Q_i^\eps&=\mathcal{T}_i^\eps\cup D_i^\eps\cup C_i^\eps\cup K_i^\eps,\\
		C_i^\eps&:=B_{\eta_{i,\eps}/4}(x_i)\setminus \mathcal{T}_i^\eps,\\
		D_i^\eps&:=B_{\eta_{i,\eps}/2}(x_i)\setminus B_{\eta_{i,\eps}/4}(x_i),\\
		K_i^\eps&:= Q_i^\eps\setminus B_{\eta_{i,\eps}/2}(x_i).		
	\end{align*}
	We then construct the corrector function as follows: for $k=1,2,3$, let $w_k,q_k$ be the solutions to the Stokes problem
	\begin{equation}
		\begin{aligned}
			-\Delta w_k+\nabla q_k&=0 \quad \text{in  }\R^3\setminus\mathcal{T},\\
			\divg(w_k)&=0\quad \text{in }\R^3\setminus\mathcal{T},\\
			w_k&=e_k\quad \text{on }\partial\mathcal{T},
		\end{aligned}
	\end{equation}
	where $e_k$ is the $k$-th unit vector in $\R^3$. Then we set
	\begin{align*}
		w_k^\eps(x)=e_k-w_k\left(\frac{x-x_i^\eps}{\eps^\alpha}\right), \: q_k^\eps=-\eps^{-\alpha}q_k\left(\frac{x-x_i^\eps}{\eps^\alpha}\right)\quad\text{in }C_i^\eps,\\
		-\Delta w_k^\eps+\nabla q_k^\eps=0,\: \divg(q_k^\eps)=0\quad\text{in } D_i^\eps,\\
		w_k^\eps=e_k,\:q_k^\eps=0\quad\text{in }\left(\R^3\setminus\bigcup_i B_{\eta_{i,\eps}/2}(x_i^\eps)
        \right).
	\end{align*}
	Note that for the definition inside $D_i^\eps$, we have to complement the PDE by suitable inhomogeneous boundary conditions, given by the definitions on $C_i^\eps$ and $K_i^\eps$. Moreover, we observe that by the choice of $\eta_{i,\eps}$, the sets $D_i$, being contained in $B_{d_i^\eps/2}$, do not overlap for different $i$. Therefore, this is well-defined. (Recall that we assumed that $\mathcal{T}\subset B_{1/4}$, so that $T_i^\eps\subset B_{\eta_i^\eps/2}(x_i^\eps)$.)
    
    \noindent In the following, we will write $w^\eps$ and $q^\eps$ for the matrix- respectively vector-valued functions that have $w_k^\eps$ respectively $q_k^\eps$ as their columns respectively entries. Moreover, we will write $A\leqc B$ whenever $A\leq CB$ for some constant $C$ depending only on $\la$, the reference particle, the diameter, the $L^\infty$-norm of $\rho$, and possibly some exponent $p$ involved in the estimate.
    
	In the following Lemma, we summarize important properties of this corrector. The proof is identical to the one in \cite[Lemma 2.1]{Richard.23}, except that we do not have to consider a test function in $H^2$ yet since we only have a finite number of holes here.
	\begin{lemma} Let $(x_i^\eps)_i\in\mathcal{A}_{1,\alpha}^N$ be a given configuration. The functions $w^\eps, q^\eps$ satisfy
		\begin{enumerate}[label=(\roman*)]
			\item $w^\eps\in W^{1,\infty}_0(\Omega_\eps;\R^{3\times 3})$, $q^\eps\in L^\infty(\Omega_\eps;\R^3)$, $\divg(w^\eps)=0$ and
			\begin{align}\label{eq:corrector w1,infty estimate}
				\norm{w^\eps}_{L^\infty(\R^3)}+\eps^\alpha\left(\norm{\nabla w^\eps}_{L^\infty(\R^3)}+\norm{q^\eps}_{L^\infty(\R^3)}\right)\leqc 1
			\end{align}
			\item For all $1\leq p<3$ and every $i=1,\dots,N$ it holds that
			\begin{align}\label{eq:corrector-Id Lp estimate p<3}
				\norm{w^\eps-\id}_{L^p(B_{\eta_i/2(x_i^\eps))}}^p\leqc \eps^{\alpha p}\eta_{i,\eps}^{3-p}
			\end{align}
			Additionally,
			\begin{align}\label{eq:corrector-Id L3 estimate}
				\norm{\id-w^\eps}_{L^3(B_{\eta_{i,\eps}/2}(x_i^\eps)}^3&\leqc\eps^{3\alpha}|\log(\eps)|,\\
				\label{eq:corrector gradient L2 estimate ganzraum}
				\norm{\nabla w^\eps}_{L^2(B_{\eta_{i,\eps}/2}(x_i^\eps))}^2+\norm{q^\eps}_{L^2(B_{\eta_{i,\eps}/2}(x_i^\eps))}^2&\leqc \eps^{\alpha},\\\label{eq:corrector gradient L1 estimate ganzraum}
				\norm{\nabla w^\eps}_{L^1(B_{\eta_{i,\eps}/2}(x_i^\eps))}+\norm{q^\eps}_{L^1(B_{\eta_{i,\eps}/2}(x_i^\eps))}&\leqc \eps^{\alpha}\eta_{i,\eps}.
			\end{align}
			\item For all $\phi\in H^1_0(\Omega_\eps)$
			\begin{align}\label{eq:corrector gradient L1 estimate Omega_eps}
				\norm{\phi |\nabla w^\eps|^{1/2}}_{L^2(\R^3)}^2+\norm{\phi |q^\eps|^{1/2}}_{L^2(\R^3)}^2&\leqc \sum_{i=1}^N\eta_{i,\eps}\norm{\nabla\phi}_{L^2(\betai)}^2.
			\end{align}
		\end{enumerate}
	\end{lemma}
	\begin{proof}
	    \textit{Step 1: Pointwise estimates and proof of $(i)$.}

        \noindent In the domains $C_i^\eps\cup D_i^\eps$ we have the following pointwise estimates:
        \begin{align}\label{eq:pointwise estimate corrector}
            |w^\eps-\id|(x)&\leqc \frac{\eps^\alpha}{|x-x_i^\eps|}\quad\text{in }C_i^\eps\cup D_i^\eps,\\ \label{eq:pointwise estimate grad corrector}
            (|\nabla w^\eps|+|q^\eps|)(x)&\leqc \frac{\eps^\alpha}{|x-x_i^\eps|^2}\quad\text{in }C_i^\eps\cup D_i^\eps,
        \end{align}
        Indeed, the estimates in $C_i^\eps$ follow from standard estimates for the Stokes equation in exterior domains (see \cite[Theorem V.3.2]{galdi2011introduction}), and then the estimate in $D_i^\eps$ is deduced from the estimate being true at the boundary $\partial D_i^\eps$ and regularity theory for the Stokes equation, see \cite[Section IV]{galdi2011introduction} for $L^p$ theory and combine this with the Sobolev embedding. Using these estimates and the fact that $0\in \mathring{\mathcal{T}}$ by assumption (so that $B_{\delta\eps^\alpha}(x_i^\eps)\subseteq \mathcal{T}_i^\eps$ for some small $\delta>0$), we can deduce $(i)$.

    \textit{Step 2: Proof of $(ii)$. } First, we use the pointwise estimate \eqref{eq:pointwise estimate corrector} to compute that in one ball $B_{\eta_{i,\eps}/2(x_i^\eps)}$ and for $p<3$
        \begin{equation}
            \norm{w^\eps-\id}_{L^p(B_{\eta_{i,\eps}/2}(x_i^\eps))}^p\leqc \eps^{\alpha p}\int_{B_{\eta_{i,\eps}}}|x-x_i^\eps|^{-p}\dx\leqc \eps^{\alpha p}\eta_{i,\eps}^{3-p},
        \end{equation}
        which shows the first estimate.
        
        The other two estimates are proven in the same way, using additionally that since $B_{\delta\eps^\alpha}(x_i^\eps)\subseteq \mathcal{T}_i^\eps$ for some $\delta$ independent of $\eps$, we have $w^\eps(x)=0$ for $|x-x_i^\eps|<\delta\eps^\alpha$:
                \begin{equation}
            \norm{w^\eps-\id}_{L^3(B_{\eta_{i,\eps}}(x_i^\eps))}^3\leqc \eps^{3\alpha }\int_{B_{\eta_{i,\eps}}(x_i^\eps)\setminus B_{\delta\eps^\alpha}(x_i^\eps)}|x-x_i^\eps|^{-3}\dx\leqc \eps^{3\alpha }\log(\eps).
        \end{equation}
        and
        \begin{equation}
            \norm{\nabla w^\eps}_{L^2(B_{\eta_{i,\eps}}(x_i^\eps))}^2\leqc \eps^{2\alpha }\int_{B_{\eta_{i,\eps}}(x_i^\eps)\setminus B_{\delta\eps^\alpha}(x_i^\eps)}|x-x_i^\eps|^{-4}\dx\leq C\eps^{2\alpha }\frac{\delta}{\eps^\alpha}.
        \end{equation}
        
        \textit{Step 3: Proof of $(iii)$. } Again, we first only consider a single ball and to alleviate notation we assume without loss of generality that $x_i^\eps=0$. Using the pointwise estimate \eqref{eq:pointwise estimate grad corrector} and with $\delta>0$ as above, we have for every $x\in C_i^\eps\cup D_i^\eps$
        \begin{equation}
            |\nabla w^\eps(x)\phi(x)^2|\leqc \frac{\eps^\alpha}{|x|^2}|\phi(x)|^2\leqc \frac{\eps^\alpha}{|x|^2}\left(\int_{\delta\eps^\alpha}^{|x|}\left|\nabla\phi\left(\frac{tx}{|x|}\right)\right|\dif{t}\right)^2.
        \end{equation}
        Therefore, we have after integrating
        \begin{align}
            &\norm{|\nabla w^\eps|\phi^2}_{L^1(B_{\eta_{i,\eps}/2}(x_i^\eps))}\leqc \eps^\alpha\int_{S^2}\int_{\delta\eps^\alpha}^{\eta_{i,\eps}/2}r^2|\phi(rn)|^2\dif{r}\dif{n}\\
            &\leqc \eps^\alpha\int_{S^2}\int_{\delta\eps^\alpha}^{\eta_{i,\eps}/2}\left(\int_{\delta\eps^\alpha}^{|x|}|\nabla\phi(tn)|\dif{t}\right)^2\dif{r}\dif{n}\\
            &\leqc \eta_{i,\eps}\eps^\alpha\int_{S^2}\int_{\delta\eps^\alpha}^{\eta_{i,\eps}/2}r^2|\nabla\phi(rn)|^2\dif{r}\dif{n}\cdot \left(\int_{\delta\eps^\alpha}^{\eta_{i,\eps}}\frac{1}{r^2}\dif{r}\right)\leqc \eta_{i,\eps}\norm{\nabla\phi}_{L^2(\betai)}^2.
        \end{align}
        The result then follows from summing over all balls and finally doing the same estimate with $q^\eps$.
	\end{proof}
    The interesting changes are in the next Lemma. For this, we introduce the parameter $\la>0$ and define the slightly enlarged cubes as before:
	\begin{align}
		\widetilde{Q_{i,\eps}}=x_i^\eps+[-\eps^{1-\la}/2,\eps^{1-\la}/2].
	\end{align}
    \begin{lemma}\label{lemma:right-hand side of stokes for corrector}
		Let $(x_i^\eps)_{i=1,\dots,N}\in\mathcal{A}_{1,\alpha}^N\cap \mathcal{B}_{\la}^N$ be a given configuration. We can write
		\begin{align}
			-\Delta w^\eps+\nabla q^\eps=\eps^{\alpha-3}M_\eps-\gamma_\eps
		\end{align}
		for some $M_\eps,\gamma_\eps\in W^{-1,\infty}(\R^3;\R^{3\times 3})$, where $\langle \gamma_\eps,v\rangle=0$ for all $v\in H^1_0(\Omega_\eps;\R^{3\times 3})$ and, for all $\psi\in H^1(\R^3;\R^3)$, we have
		\begin{equation}\label{eq:estimate for M_eps}
			\begin{aligned}
				|\langle M_\eps-\rho\mathcal{R},\psi\rangle|\leqc &\left(W_2(\rho_\eps,\rho)+\eps^{1-\la}\right)\norm{\psi}_{H^1(\R^3)}\\
				&+\sum_{i=1}^N \left(\eta_{i,\eps}^{-1/2}\eps^{3}\norm{\psi}_{H^1(\widetilde{Q_i^\eps})}+\eta_{i,\eps}^{-1}\eps^\alpha\norm{\psi}_{L^2(Q_i^\eps)}\right)
			\end{aligned}
		\end{equation}
	\end{lemma}
	\begin{proof}
		We still follow the main ideas of \cite[Lemma 2.2]{Richard.23}; the changes lie in the estimate (2.14). In this proof, we will often write $\eta_i$ instead of $\eta_{i,\eps}$ to abbreviate notation.\\
		We have that $-\Delta w^\eps+\nabla q^\eps$ is supported on $\bigcup_i\partial D_i^\eps\cup\partial\Omega_\eps$, and we define $\gamma_\eps$ to be the part that is supported on $\partial\Omega_\eps$, which consequently satisfies $\langle \gamma_\eps,v\rangle=0$ for all $v\in H^1_0(\Omega_\eps)$. Then, the matrix $M_\eps$ has the columns
		\begin{align}
			M_{\eps,k}=\eps^{3-\alpha}\sum_{i=1}^N m_{k,i}^\eps + \divg(\textbf{1}_{D_i^\eps}(q_k^\eps\id-\nabla w_k^\eps)),
		\end{align}
		where, using \cite[Lemma 2.3.5]{Allaire.90a},
		\begin{align}
			m_{k,i}^\eps&=\eps^{-\alpha}(q_k\id-\nabla w_k)(\eps^{-\alpha}x)\mathcal{H}^2|_{\partial B_{\eta_{i,\eps}/4}(x_i^\eps)}\\
            &=\frac{\eps^{\alpha}}{2}\left(\mathcal{R}_k+3(\mathcal{R}_k\cdot n)n+\eta_i^{-1}\eps^\alpha r_{k,i}^\eps\right)\delta_{\eta_i/4}^i,\quad &\text{where } \delta_{\eta_i/4}^i=\frac{\mathcal{H}^2|_{\partial B_{\eta_i/4}(x_i)}}{|\partial B_{\eta_i/4}|},\\
			&\norm{r_{k,i}^\eps}_{W^{1,\infty}(\partial B_{\eta_i/4})}\leqc 1\nonumber
		\end{align}
		Before we estimate these expressions, we note a few inequalities that we will use: by the Sobolev embedding and a scaling argument we have that for any $\psi\in H^1(Q_i^\eps)$
        \begin{align}\label{eq:Sobolev embedding with average}
            \norm{\psi-\fint_{Q_i^\eps} \psi}_{L^6(Q_i^\eps)}\leqc\norm{\nabla \psi}_{L^2(Q_i^\eps)}
        \end{align}
        Similarly, using also the trace inequality, we have the Poincaré-like inequality
        \begin{align}\label{eq:weird poicare like ineq}
            \fint_{\del B_{\eta_i/4}(x_i^\eps)} \left|\psi-\fint_{B_{\eta_i/4}(x_i^\eps)}\psi\dif{y}\right|\dx\leqc \eta_i^{-1/2}\norm{\nabla \psi}_{L^2(B_{\eta_i/4}(x_i^\eps))}.
        \end{align}
        We can now begin estimating the terms involved in $M_k^\eps$. As in \cite[Equations (2.15) and (2.16)]{Richard.23}, we can show that
		\begin{align}
			\eps^{3-\alpha}\left|\int_{\R^3}\psi\sum_i \divg(\textbf{1}_{D_i^\eps}(q_k^\eps\id-\nabla w_k^\eps))dx\right|\leqc \sum_{i=1}^N \eta_i^{-1/2}\eps^{3}\norm{\psi}_{H^1(B_{\eta_{i}/2}(x_i^\eps))},\label{eq:2.15 richard}\\
			\langle\eps^3\sum_i r_{k,i}^\eps\delta_{\eta_i/4}^i,\psi\rangle\leqc\sum_{i=1}^N\eps^{3/2}\norm{\psi}_{L^2(Q_i^\eps)}+\eta_i^{-1/2}\eps^{3}\norm{\nabla\psi}_{L^2(Q_i^\eps)}.\label{eq:2.16 richard}
		\end{align}
		Indeed, in order to prove \eqref{eq:2.15 richard} we use again the pointwise estimates \eqref{eq:pointwise estimate grad corrector} to deduce that
        \begin{align*}
            \eps^{3-\alpha}\left|\int_{\R^3}\sum_i\divg(\textbf{1}_{D_i^\eps}(q_k^\eps\id-\nabla w_k^\eps))\psi \dx\right|&\leqc \eps^{3-\alpha}\sum_i\eta_i^{3/2}\norm{q_k^\eps\id-\nabla w_k^\eps}_{L^\infty(D_i^\eps)}\norm{\psi}_{H^1(D_i^\eps)}\\
            &\leqc\sum_i\eta_i^{3/2-2}\eps^{3-\alpha+\alpha}\norm{\psi}_{H^1(B_{\eta_i/2}(x_i^\eps))}.
        \end{align*}
        For \eqref{eq:2.16 richard}, we use \eqref{eq:Sobolev embedding with average} and
        \eqref{eq:weird poicare like ineq} along with the Hölder inequality to estimate the average of $\psi\in H^1(Q_i^\eps)$ as follows:
        \begin{align*}
            \left|\fint_{\partial B_{\eta_i/4}(x_i^\eps)}\psi\dx\right|&\leqc\fint_{\partial B_{\eta_i/4}(x_i^\eps)}|\psi-\fint_{B_{\eta_i/4}(x_i^\eps)}\psi\dif{y}|\dx+\fint_{ B_{\eta_i/4}(x_i^\eps)}|\psi-\fint_{Q_i^\eps}\psi\dif{y}|\dx+|\fint_{Q_i^\eps}\psi\dif{y}|\\
            &\leqc\eta_i^{-1/2}\norm{\nabla \psi}_{L^2(Q_i^\eps)}+\eps^{-3/2}\norm{\psi}_{L^2(Q_i^\eps)}.
        \end{align*}
        Therefore, after summing over all cubes,
        \begin{align}
            |\langle\eps^3\sum_i r_{k,i}^\eps\delta_{\eta_i/4}^i\phi,\psi\rangle|&\leqc \eps^{3}\sum_{i=1}^N\eta_i^{-1/2}\norm{\nabla\psi}_{L^2(Q_i^\eps)}+\eps^{-3/2}\norm{\psi}_{L^2(Q_i^\eps)}\\
            &\leqc\sum_{i=1}^N \eps^{3/2}\norm{\psi}_{L^2(Q_i^\eps)}+\eta_i^{-1/2}\eps^{3}\norm{\nabla\psi}_{L^2(Q_i^\eps)}.
        \end{align}
        Now it remains to estimate the distance between $\rho \mathcal{R}_k$ and
		\begin{align*}
        \widetilde{\rho_\eps}:=\frac{\eps^3}{2}\sum_{i=1}^N (\mathcal{R}_k+3(\mathcal{R}_k\cdot n)n)\delta_{\eta_i/4}^i.
		\end{align*}
		To do this, we consider a ''smoothened" empirical density
		\begin{align*}
			\overline{\rho}_\eps =\sum_{i=1}^N \frac{\eps^3}{\eps^{3(1-\la)}}\textbf{1}_{\widetilde{Q_i^\eps}}\in\mathcal{P}(\R^3).
		\end{align*}
		In the periodic setting of \cite{Richard.23} (where we only consider $\la=0$), this $\overline{\rho}_\eps$ is simply the constant function $1$.\\
		It now remains to estimate
		\begin{align}\label{eq: first estimate for H^-1 lemma 2}
			\langle (\rho\mathcal{R}_k-\widetilde{\rho}_\eps),\psi\rangle\leq \langle(\overline{\rho}_\eps\mathcal{R}_k-\widetilde{\rho}_\eps),\psi\rangle+\norm{\rho\mathcal{R}_k-\overline{\rho}_\eps\mathcal{R}_k}_{H^{-1}(\R^3)}\norm{\psi}_{H^1(\R^3)}.
		\end{align}The first term can be estimated in the same way as in \cite[Equation (2.18) and below]{Richard.23}: we observe that by rotation symmetry
        \begin{align}
            \mathcal{R}_k=\fint_{\del B_{\eta_i/4}}\frac{1}{2}(\mathcal{R}_k+3(\mathcal{R}_k\cdot n)n)\dx.
        \end{align}
        Using this equality and the estimates \eqref{eq:Sobolev embedding with average} and \eqref{eq:weird poicare like ineq}, we can show that for $\psi\in H^1(Q_i^\eps)$
		\begin{align*}
			&|\langle\left(\frac{\eps^3}{\eps^{3(1-\la)}}\textbf{1}_{\widetilde{Q_i^\eps}}\mathcal{R}_k-\frac{\eps^3}{2} (\mathcal{R}_k+3(\mathcal{R}_k\cdot n)n)\delta_{\eta_i/4}^i\right),\psi\rangle|\\
			&=\frac{\eps^3}{2}\left|\fint_{\partial B_{\eta_i/4}(x_i^\eps)}(\psi-\fint_{\widetilde{Q_i^\eps}}\psi)\cdot(\mathcal{R}_k+3(\mathcal{R}_k\cdot n)n)\dx\right|\\
            &\leqc\left(\eta_i^{-1/2}\eps^3\norm{\nabla \psi}_{L^2(B_{\eta_i/4}(x_i^\eps))}+\eta_i^{-1/2}\eps^3\norm{\psi-\fint_{\widetilde{Q_i^\eps}}\psi}_{L^6(B_{\eta_i/4}(x_i^\eps))}\right)\\
			& \leqc \eta_i^{-1/2}\eps^3\norm{\nabla \psi}_{L^2(\widetilde{Q_i^\eps})}.
		\end{align*}
		Therefore,
		\begin{align*}
			\langle(\overline{\rho}_\eps\mathcal{R}_k-\widetilde{\rho}_\eps),\psi\rangle&\leqc \sum_{i=1}^N \eta_i^{-1/2}\eps^3\norm{\nabla \psi}_{L^2(\widetilde{Q_i^\eps})},
		\end{align*}
		which bounds the first term in \eqref{eq: first estimate for H^-1 lemma 2}.
        
		For the second term, we to make a connection between the $H^{-1}$-norm and the $2$-Wasserstein distance, for which we use Proposition \ref{prop:W2 dist L^infty ineq}, the fact that $\mathcal{R}_k$ is bounded and the triangle inequality to deduce that 
		\begin{align}\label{eq: second estimate for H^-1 lemma 2}
			\norm{\overline{\rho}_\eps\mathcal{R}_k-\rho\mathcal{R}_k}_{H^{-1}(\R^3)}&\leqc \max\{\norm{\overline{\rho}_\eps}_{L^\infty},\norm{\rho}_{L^\infty}\}^{1/2}W_2(\overline{\rho}_\eps,\rho)\nonumber\\
			&\leqc W_2(\overline{\rho}_\eps,\rho_\eps)+W_2(\rho_\eps,\rho).
		\end{align}
		Note that we used that since $(x_i^\eps)_i\in \mathcal{B}_{\la}^N$, $\norm{\overline{\rho}_\eps}_{L^\infty}$ is bounded by $16\norm{\rho}_\infty$.
        
		Now, in order to estimate $W_2(\overline{\rho}_\eps,\rho_\eps)$. For this, we can explicitly write down the following transport plan:
		\begin{align*}
			\gamma=\frac{1}{N}\sum_{i=1}^N\delta_{x_i^\eps}\otimes\frac{\textbf{1}_{\widetilde{Q_i^\eps}}}{|\widetilde{Q_i^\eps}|}
		\end{align*}
		Since the cubes have diameter $\sqrt{3}\eps^{1-\la}$, we have that
		\begin{align*}
			\supp(\gamma)\subs \{(x,y)\in\R^3: |x-y|\leq \sqrt{3}\eps^{1-\la}\},
		\end{align*}
		so that
		\begin{align*}
			W_2(\overline{\rho}_\eps,\rho_\eps)\leq \sqrt{3}\eps^{1-\la}.
		\end{align*}
        Finally to get the statement of the theorem as asserted, we note that $\betai\subseteq Q_i^\eps\subseteq\widetilde{Q_i^\eps}$, so that we can simply bound $\norm{\nabla \psi}_{L^2(\betai)}\leq \norm{\nabla \psi}_{L^2(Q_i^\eps)}\leq\norm{\nabla \psi}_{L^2(\widetilde{Q_i^\eps})}$.
	\end{proof}
	Finally, also Lemma 2.3 in \cite{Richard.23} remains basically unchanged.
	\begin{lemma}
		Let $(x_i^\eps)_i\in \mathcal{A}_{1,\alpha}$ be a given configuration. For all $1<p<\infty$, there exists a  linear operator $B_\eps: W^{1,p}(\R^3;\R^3)\rightarrow W^{1,p}_0(\Omega_\eps;\R^3)$ such that for all $\phi\in W^{1,p}(\R^3;\R^3)$  that satisfy $\divg(\phi)=0$ we have
        \begin{align}
            \divg(B_\eps(\phi))=w^\eps:\nabla\phi
        \end{align}
        and
		\begin{align}\label{eq:estimate Bogovskii gradient}
			\norm{\nabla B_\eps(\phi)}_{L^p(\R^3)}^p&\leqc \sum_{i=1}^N \norm{(I-w^\eps):\nabla\phi}_{L^p(\betai)}^p,\\\label{eq:estimate Bogovskii} \norm{B_\eps(\phi)}_{L^p(\R^3)}^p&\leqc\sum_{i=1}^N \eta_i^p \norm{(I-w^\eps):\nabla\phi}_{L^p(\betai)}^p.
		\end{align}
	\end{lemma}
    \begin{proof}
        Let $\phi\in W^{1,p}(\R^3;\R^3)$ be a divergence-free function (so that $\id:\nabla\phi=0$). Note that outside $A_i^\eps:=C_i^\eps\cup D_i^\eps$, $w^\eps=\in\{0,\id\}$, so that
        \begin{align}
            w^\eps:\nabla\phi=0\quad\text{in }\R^3\setminus\bigcup_i A_i^\eps.
        \end{align}
        Inside $A_i^\eps$, we observe that since $w^\eps$ is divergence free and $w^\eps=\id$ on $\del B_{\eta_{i,\eps}/2}(x_i^\eps)=\del(A_i^\eps\cup \mathcal{T}_i^\eps)$,
        \begin{align}
            \int_{A_i^\eps}w^\eps:\nabla\phi\dx=\int_{A_i^\eps\cup \mathcal{T}_i^\eps}w^\eps:\nabla\phi\dx=\int_{A_i^\eps\cup \mathcal{T}_i^\eps}\divg((w^\eps-\id)\phi)\dx=0.
        \end{align}
        Therefore, we can use standard Bogovskii extension operators in each $A_i^\eps$ (see \cite[Chapter III.3]{galdi2011introduction}). To be precise, for each $i$ there exists an operator $B_i^\eps$ acting on $L^p_0$, the space of $L^p$ functions with zero mean value:
        \begin{align}
            &B_i^\eps:L^p_0(A_i^\eps)\rightarrow W^{1,p}_0(A_i^\eps)\quad\text{with}\\
            \label{eq: bogovskii estimate for A_i^eps}&\divg(B_i^\eps(h))=h,\quad\norm{B_i^\eps(h)}_{W^{1,p}(A_i^\eps)}\leqc \norm{h}_{L^p_0(A_i^\eps)}.
        \end{align}
        Then we construct the desired operator $B^\eps$ as
        \begin{align}
            B^\eps(\phi)=\sum_i B_i^\eps(w^\eps:\nabla\phi).
        \end{align}
        The estimates \eqref{eq:estimate Bogovskii} and \eqref{eq:estimate Bogovskii gradient} follow from summing the estimate \eqref{eq: bogovskii estimate for A_i^eps} and applying the Poicaré inequality in the domain $A_i^\eps\subseteq B_{\eta_{i,\eps}/2}(x_i^\eps)$ (so that the Poincaré constant is proportional to $\eta_
        {i,\eps}$).
    \end{proof}
    \section{Some probabilistic results}\label{sec:prop results}
    \subsection{Assuming $\alpha>2$ to avoid particle overlap}
    We first discuss the measure of the first set $\mathcal{A}_{1,\alpha}$, the set in which the holes overlap or are very close to each other. It is known (see \cite[Appendix A.3]{Hauray.2009}) that since $\alpha>2$ and $\min d_i^\eps$ scales like $N^{-2/3}=\eps^2$, the set $\mathcal{A}_{1,\alpha}^c$ has very small measure. Note that while the statement of the proposition in \cite{Hauray.2009} is on the torus, the proof of this part of the inequalities works the same in $\R^3$.
    \begin{prop}[{\cite[Appendix A.3, in particular equation (105)]{Hauray.2009}}]\label{prop:min distance scaling}
        Define the sets $\mathcal{A}_{L,\alpha}$ as in \eqref{eq:A0} and assume that $\rho\in L^\infty(\R^3)$. Then we have for all $L\leq(2\eps^3 4\pi\norm{\rho}_{L^\infty}/3)^{-1/3}/2$ that 
        	\begin{align*}
		\mathbb{P}\left[\mathcal{A}_{L,2}\right]\geq e^{-4\pi\norm{\rho}_{L^\infty} L^3/3}.
	\end{align*}
    	In particular, for $L=\eps^{\alpha-2}$, $\alpha>2$ and $\eps$ small enough,
    \begin{align*}
    \mathbb{P}\left[\mathcal{A}_{1,\alpha}^c\right]=\mathbb{P}\left[(x_i^\eps)_i\in\mathcal{A}_{\eps^{\alpha-2},2}^c\right]\leq 1-e^{-C\eps^{3(\alpha-2)}}\leq \eps^{3(\alpha-2)}\rightarrow 0.
    \end{align*}
    \end{prop}
    \noindent Here, in the last inequality we used that $(1-e^{-x})/x\leq 1$ for $x>0$. Note that here we needed $\alpha>2$ to ensure the convergence of the exponential term -- for general $\alpha>1$, the estimate can still be applied, but the estimate does not vanish as $\eps\rightarrow 0$, that is to say, it is likely that at least some of the holes overlap when $\alpha\leq 2$. Since it will usually be only few holes that overlap, one could hope that the result still holds true, but we do not pursue this direction here.
    
    Next, we consider the set $\mathcal{B}_\la^N$, where the holes are on average far away from each other (but individual holes may be close). We will want to control the overlap of these cubes, for which we will use the following result.
	\begin{prop}[{\cite[Proposition 8, in particular equation (A.2)]{Hauray_Jabin.2015}}]\label{prop:control overlap with quantified prob}
		Define the enlarged cubes $\widetilde{Q_{i,\eps}}$ as in \eqref{eq:q_i tilde def} and let $\rho\in L^\infty(\R^3)$ have compact support. We then have that the overlap of these larger cubes, normalized by the expected number of overlaps, is bounded with a overwhelming probability. To be precise, it holds that
		\begin{align*}
			\mathbb{P}[(\mathcal B_\la^N)^c]\leq \mathbb{P}\left[\frac{\norm{\sum_{i=1}^N \mathbf{1}_{\widetilde{Q_{i,\eps}}}}_{L^\infty(\R^3)}}{N\eps^{3(1-\la)}}\geq 16\norm{\rho}_{L^\infty(\R^3)}\right]\leq C\eps^{-3(1-\la)}e^{-8(2\ln (2)-1)\norm{\rho}_\infty\eps^{-3\la}}\rightarrow 0,
		\end{align*}
        where $C$ depends only on the diameter of the support of $\rho$.
	\end{prop}
    Note that in particular for the cubes with $\la=0$, we have for any $\la>0$
	\begin{align}
		\mathbb{P}\left[\norm{\sum_{i=1}^N \mathbf{1}_{x_i^\eps+[-\eps/2,\eps/2]^3}}_{L^\infty(\R^3)}\leq 16\norm{\rho}_\infty\eps^{-3\la}\right]\geq\mathbb{P}[(x_i^\eps)_i\in\mathcal{B}_{\la}^N]\geq 1- C\eps^{-3(1-\la)}e^{-C'\eps^{-3\la}}\rightarrow 1.
	\end{align}
   Note that the assumption on the support of $\rho$ is not necessary to have that $\mathcal{B}_\la^N$ has overwhelming measure (see \cite{GINE2002907} or \cite{Einmahl_2005} for results in this spirit, though they do not provide convergence estimates).
   \subsection{Estimating the expectation of $d_i^\eps$ and related quantities}
   To prove our main results, we naturally want to put together all the estimates involving sums over $\eta_{i,\eps}$ and norms over the cubes or balls, like the ones in \eqref{eq:corrector gradient L1 estimate Omega_eps}, \eqref{eq:estimate for M_eps}, \eqref{eq:estimate Bogovskii gradient} and \eqref{eq:estimate Bogovskii}. It is well-known that for fixed $i$, the expected minimal distance $d_i^\eps$ from  $x_i^\eps$ to the $N-1$ other random points scales like $N^{-1/3}=\eps$.  This can formally be justified by the following reasoning: since the probability of a point $x_j$ ($j\neq i$) being in a ball of radius $r$ scales like $r^3$, we have for large $N$
    \begin{align}
        \mathbb{P}[d_i^\eps>r]\sim(1-cr^3)^{N-1}\sim e^{-cr^3(N-1)},
    \end{align}
    which is exactly of order $1$ if $r^3\sim N$, that is, $r\sim N^{-1/3}=\eps$.
    
    Therefore, we expect that we can bound these sums involving (powers of) $\eta_{i,\eps}$ by an expression involving $\eps$. To be more precise, we choose
	\begin{align*}
		\eta_{i,\eps}=\min\{m_\eta\eps^\beta,d_i^\eps\}
	\end{align*}
	for some $\beta\in [1,\alpha]$ to be chosen later, $m_\eta\leq 1$ so that still $\eta_i\geq \eps^\alpha$, and want an estimate of the following spirit:
	\begin{align}
		\sum_{i=1}^N \eta_{i,\eps}^\kappa\norm{\phi}_{L^p(Q_i^\eps)}^p\leqc \eps^{-3\lambda}(\eps^{\beta})^\kappa\norm{\phi}_{L^p(\R^3)}^p\quad\text{ for any }\kappa\in(-3,\infty),
	\end{align}
	where the factor $\eps^{-3\la}$ represents a small loss in the convergence rate due to the potential overlap of the cubes $Q_i^\eps$ which is controlled for realizations in  $\mathcal{B}_{\la}^N$. Indeed, for $\kappa\geq 0$ this is easy to show since $\eta_i\leq\eps^\beta$ and the sum over the norms can be bounded up to  arbitrarily small loss of convergence (assuming we have a configuration in $\mathcal{B}_{\la}$).
    
    The case for negative $\kappa$ is more interesting. In fact, since we will only need this statement with negative exponents for $\phi\in H^2(\R^3)\subs L^\infty(\R^3)$ (in the end, we only need this estimate with negative exponents where $\phi$ is some derivative of $u$, which is regular), we will only estimate the expectation in the following lemma.
	\begin{lemma}\label{lemma:putting together estimates with eta_i}
		Set $\eta_{i,\eps}=\min\{m_\eta\eps^\beta, d_i^\eps\}$ for some $m_\eta\leq 1$ and $\beta\in [1,\alpha]$, with $m_\eta=1$ if $\beta=\alpha$.
        Then, for $\kappa\in(-3,\infty)$,
            \begin{align}\label{eq:expectation eta_i}
            \mathbb{E}\left[\eta_{i,\eps}^\kappa\right]&\leqc m_\eta^\kappa(1+\eps^{3(\beta-1)}) \eps^{\beta\kappa}.
			\end{align}
			Moreover, let $\psi\in L^2(\R^3)$ and $\kappa\in[0,\infty)$. Then
			\begin{align}\label{eq:sum over eta estimate for L^2 fcts and pos kappa}
            \mathbb{E}\left[ \textbf{1}_{\mathcal{A}_{1,\alpha}^N\cap \mathcal{B}_\la^N}\sum_{i=1}^N \eta_{i,\eps}^{\kappa}\norm{\psi}_{L^2(\widetilde{Q_i^\eps})}^2\right]&\leqc m_\eta^\kappa\eps^{\beta\kappa-3\lambda}\norm{\psi}_{L^2(\R^3)}^2,\\\label{eq:sum over eta estimate for L^2 fcts and pos kappa2}
				\mathbb{E}\left[ \textbf{1}_{\mathcal{A}_{1,\alpha}^N\cap \mathcal{B}_\la^N}\sum_{i=1}^N \eta_{i,\eps}^{\kappa}\norm{\psi}_{L^2(B_{\eta_{i,\eps}/2}(x_i^\eps))}^2\right]&\leqc m_\eta^\kappa\eps^{\beta\kappa}\norm{\psi}_{L^2(\R^3)}^2.
			\end{align}
	\end{lemma}
	\begin{proof}
		To estimate the expectation of $\eta_{i,\eps}^\kappa$, we observe that for $\kappa\geq 0$ can simply estimate $\eta_{i,\eps}^\kappa\leq m_\eta^\kappa\eps^{\beta\kappa}$, which gives us the estimates immediately. For $\kappa\in(-3,0)$, we note that we always have $\eta_{i,\eps}^\kappa\geq m_\eta^\kappa \eps^{\beta\kappa}$. Then, we use the layer cake representation:
		\begin{align*}
			\mathbb{E}\left[\eta_{i,\eps}^\kappa\right]=\int_0^\infty \mathbb{P}\left[\eta_{i,\eps}^\kappa\geq t\right]\dt =\int_{m_\eta^\kappa\eps^{\beta\kappa}}^\infty \mathbb{P}\left[d_{i,\eps}^\kappa\geq t\right]\dt + m_\eta^\kappa\eps^{\beta\kappa}.
		\end{align*}
		Now we see that since the $x_i^\eps$ are i.i.d. with bounded density $\rho$, so we can bound (recall that $\kappa<0$)
		\begin{align*}
			\mathbb{P}\left[d_{i,\eps}^\kappa\geq t\right]=\mathbb{P}\left[d_{i,\eps}\leq t^{1/\kappa}\right]=\mathbb{P}\left[\exists j\neq i:|x_i^\eps-x_j^\eps|\leq t^{1/\kappa}\right]\leqc Nt^{3/\kappa}.
		\end{align*}
		Therefore, using that $m_\eta\leq 1$, so that $m_\eta^3\leq 1$, we have
		\begin{align*}
			\mathbb{E}\left[\eta_{i,\eps}^\kappa\right]\leqc m_\eta^\kappa\eps^{\beta\kappa}+\int_{m_\eta^\kappa\eps^{\beta\kappa}}^\infty Nt^{3/\kappa}\dt\leqc m_\eta^\kappa \eps^{\beta\kappa}+m_\eta^{3+\kappa}\eps^{-3}\eps^{3\beta+\beta\kappa}\leqc m_\eta^\kappa(1+\eps^{3(\beta-1)})\eps^{\beta\kappa}.
		\end{align*}
		Then the equations \eqref{eq:sum over eta estimate for L^2 fcts and pos kappa} and \eqref{eq:sum over eta estimate for L^2 fcts and pos kappa2} follow by estimating that $\eta_{i,\eps}\leq m_\eta\eps^\beta$ for $\kappa\geq 0$, and using that for $(x_i^\eps)_i\in\mathcal{B}_{\la}^N$,
		\begin{align*}
			\sum_{i=1}^N \norm{\psi}_{L^2(\widetilde{Q_i^\eps})}^2\leqc m_\eta^\kappa\eps^{-3\la}\norm{\psi}^2_{L^2(\R^3)},
		\end{align*}
		whereas for the second inequality we note that by construction of $\eta_{i,\eps}$, the sets $B_{\eta_{i,\eps}/2}(x_i^\eps)$ are disjoint, so that
		\begin{align*}
			\sum_{i=1}^N \norm{\psi}_{L^2(B_{\eta_{i,\eps}/2}(x_i^\eps))}^2\leq \norm{\psi}^2_{L^2(\R^3)}.
		\end{align*}
	\end{proof}
    \noindent Finally, of course we need to estimate the Wasserstein distance between the empirical measure $\rho_\eps$ and $\rho$. In the most general case, we can get a convergence rate of $\eps^{3/4}=N^{-1/4}$:
    \begin{prop}[{\cite[Theorem 1]{fournier2013rateconvergencewassersteindistance}}] \label{prop:expected Wasserstein distance}
        Let $\mu$ be a probability measure on $\R^d$ with bounded $q$-th moment for some $q>4$ (which includes compactly supported measures) and let $(X_i)$ be a sequence of i.i.d. $\mu$-distributed measures. Then, defining $\mu_N=\frac{1}{N}\sum_{i=1}^N \delta_{X_i}$ to be the empirical measure,
        \begin{align}
            \mathbb{E}\left[W_2(\mu,\mu_N)^2\right]\leqc N^{-1/2}=\eps^{3/2}.
        \end{align}
    \end{prop}
    \noindent Note that this can be improved to $\eps^2$ (for the squared Wasserstein distance) if one makes some assumptions on $\rho$ (for example, if $\rho=\textbf{1}_{[0,1]^3}$, see \cite{fournier2013rateconvergencewassersteindistance} or \cite{divol2021shortproofrateconvergence}).
        \section{Proof of Theorem \ref{thm:critical case} and \ref{thm:subcritical case}}
    We now turn our attention to proving Theorem \ref{thm:critical case} and \ref{thm:subcritical case}.
    
	For this, we can put together the estimates as in \cite[Proposition 3.1]{Richard.23}, with modifications mainly coming from the estimates in Lemma \ref{lemma:right-hand side of stokes for corrector} and \ref{lemma:putting together estimates with eta_i}. 
    \begin{defi}
        Define for a configuration $x_i^\eps$ in $\mathcal{A}_{1,\alpha}^N\cap \mathcal{B}_{\la}$
	\begin{align*}
		\check u_\eps&:=w^\eps u-B_\eps(u),\\
		v_\eps&:=\check u_\eps-u_\eps.
	\end{align*}
    If the we do not have a configuration in $\mathcal{A}_{1,\alpha}^N\cap \mathcal{B}_{\la}^N$, we set $v_\eps=0$.
    \end{defi} 
    \noindent Note that we have seen in Propositions \ref{prop:control overlap with quantified prob} and \ref{prop:min distance scaling} that $\mathcal{A}_{1,\alpha}^N\cap \mathcal{B}_{\la}^N$ has overwhelming measure, so that outside of $\mathcal{A}_{1,\alpha}^N\cap \mathcal{B}_{\la}^N$ it does not really matter what we define $v_\eps$ to be.
    
	Using the following Propositions, we are in a position to prove Theorem \ref{thm:critical case} and \ref{thm:subcritical case}.
	\begin{prop}\label{prop:critical case gronwall estimate}
		Set $\eta_{i,\eps}=\min\{d_i^\eps,\eps^\beta\}$ for some $\beta\in [1,\alpha]$. Under the assumptions of Theorem \ref{thm:critical case}, there exists a constant $C\geq 1$ such that we have for any $\la>0$ and $0\leq t\leq T$
		\begin{equation}\label{eq:energy estimate critical case}
        \begin{aligned}
            \mathbb{E}&\left[\norm{v_\eps(t)}_{L^2(\Omega_\eps)}^2 +(\eps^\gamma-C\eps^{\beta})\norm{\nabla v_\eps}_{L^2((0,t)\times\Omega_\eps)}^2\right]\\
			&\leqc\mathbb{E}\left[\norm{v_\eps(0)}_{L^2(\Omega_\eps)}^2+\norm{f_\eps-f}_{L^2((0,T)\times\Omega_\eps)}^2+\norm{v_\eps}_{L^2((0,t)\times\Omega_\eps)}^2+\eps^{-\gamma}(W_2(\rho_\eps,\rho)+\eps^{1-\la})^2\right]\\
			&+(\eps^{2\alpha-3-\gamma+\beta}+\eps^{\alpha-\beta-3\la}+\eps^{2\gamma}+\eps^{2\beta}).
        \end{aligned}
		\end{equation}
	\end{prop}
    \begin{prop}\label{prop:subcritical case gronwall estimate}
		Set $\eta_{i,\eps}=\min\{d_i^\eps,m_\eta\eps^\beta\}$ for some $\beta\in[1,\alpha]$, where we choose $m_\eta=1$ if $\beta=\alpha$. Under the assumptions of Theorem \ref{thm:subcritical case}, we have for any $\la>0$ and $0\leq t\leq T$
		\begin{align}\label{eq:energy estimate subcritical case}
			\mathbb{E}&\left[\norm{v_\eps(t)}_{L^2(\Omega_\eps)}^2 +(\mu_0\eps^\gamma-C (m_\eta+\eps^{3(\beta-1)})\eps^{\beta})\norm{\nabla v_\eps}_{L^2((0,t)\times\Omega_\eps)}^2\right]\\
			\leqc&\mathbb{E}\left[\norm{v_\eps(0)}_{L^2(\Omega_\eps)}^2+\norm{f_\eps-f}_{L^2((0,T)\times\Omega_\eps)}^2+\norm{v_\eps}_{L^2((0,t)\times\Omega_\eps)}^2+\eps^{2\alpha+\gamma-6}(W_2(\rho_\eps,\rho)+\eps^{1-\la})^2\right]\\
			&+(\eps^{2\alpha-3-\gamma+\beta}+\eps^{2\alpha+\gamma-3-3\la-\beta}+\eps^{2\gamma}+\eps^{2\beta}+\eps^{2\alpha+2\gamma-6}).
		\end{align}
	\end{prop}
	\begin{proof}[Proof of Theorem \ref{thm:critical case}]
		In Proposition \ref{prop:critical case gronwall estimate}, we choose $\beta=\max\{1,\gamma\}=1$ (since $\alpha>2$, we have that $\gamma<1$), so that we can absorb the gradient term. Then we use Grönwall's inequality to deduce that for configurations in $\mathcal{A}_{1,\alpha}^N\cap B_\la^N$ 
		\begin{align*}
				\mathbb{E}\left[\norm{v_\eps(t)}_{L^2(\Omega_\eps)}^2\textbf{1}_{\mathcal{A}_{1,\alpha}^N\cap B_\la^N} \right]&\leqc\mathbb{E}\left[\norm{v_\eps(0)}_{L^2(\Omega_\eps)}^2+\norm{f_\eps-f}^2_{L^2((0,T)\times\Omega_\eps)}+\eps^{-\gamma}(W_2(\rho_\eps,\rho)+\eps^{1-\la})^2\right]\\
            &+(\eps^{\alpha-1-3\la}+\eps^{6-2\alpha}).
		\end{align*}
        Here we used that $2\beta=2>2\gamma=6-2\alpha$ and that for $\alpha\in (2,3)$, $2\alpha-3+\beta-\gamma\geq 2\alpha-3\geq \alpha-1-3\la$. Also, we note that for $\la<1/4$ due to Proposition \ref{prop:expected Wasserstein distance},
        \begin{align}
        	\mathbb{E}\left[\eps^{-\gamma}(W_2(\rho_\eps,\rho)+\eps^{1-\la})^2)\right]\leqc \mathbb{E}\left[\eps^{-\gamma}W_2(\rho_\eps,\rho)^2)\right]\leqc \eps^{3/2-\gamma}.
        \end{align}
		We then observe that due to \eqref{eq:corrector-Id Lp estimate p<3} and \eqref{eq:estimate Bogovskii} (combined with Lemma \ref{lemma:putting together estimates with eta_i}),
		\begin{align*}
			\sup_{t\in [0,T]}\norm{v_\eps-(u_\eps-u)}_{L^2(\R^3)}^2\leqc \eps^{\gamma+2\alpha-3-3\la},
		\end{align*}
		which has a higher order of convergence. Finally, we observe that for $\la<1/6$, we have $\alpha-1-3\la\geq \alpha-3/2$, allowing us to drop the power of $\alpha-1-3\la$.
        
        For configurations outside of $\mathcal{A}_{1,\alpha}^N\cap B_\la^N$ (where we set $v_\eps=0$), we note that by using Grönwall's inequality on the energy inequality \eqref{eq:energy ineq}, $u_\eps(t)$ is uniformly bounded in $L^2(\R^3)$:
        \begin{align*}
        \norm{u_\eps(t)}_{L^2(\Omega_\eps)}^2\leqc \norm{u_{0,\eps}}_{L^2(\Omega_\eps)}^2 e^{\int_0^T \norm{f_\eps(s)}_{L^2(\Omega_\eps)}^2\dif{s}}<\infty.
        \end{align*}
        Moreover, $u$ is also uniformly bounded in $L^2((0,T);L^2(\Omega_\eps))$ since it is smooth. Therefore, for realizations in $(\mathcal{A}_{1,\alpha}^N\cap B_\la^N)^c$,
        \begin{align*}
            \mathbb{E}\left[\norm{u(t)-u_\eps(t)}_{L^2(\Omega_\eps)}^2\textbf{1}_{(\mathcal{A}_{1,\alpha}^N\cap B_\la^N)^c} \right]\leqc \mathbb{P}[(\mathcal{A}_{1,\alpha}^N)^c]+\mathbb{P}[(\mathcal{B}_{\la}^N)^c].
        \end{align*}
        In order to get the estimate as stated in th theorem we then apply the estimates for $\mathcal{A}$ and $\mathcal{B}_\la$ above (Proposition \ref*{prop:min distance scaling} and \ref{prop:control overlap with quantified prob}) and note that we can absorb the for any fixed $\la>0$ exponentially fast converging term $\eps^{-3(1-\la)}e^{-c\eps^{-3\la}}$.
	\end{proof}
    \begin{proof}[Proof of Theorem \ref{thm:subcritical case}]
    We again choose in Proposition \ref{prop:subcritical case gronwall estimate} $\beta=\max\{1,\gamma\}$ and the constant $m_\eta$ by
    \begin{align*}
        m_\eta=\begin{cases}1,&\text{if }\gamma=\alpha,\\
        \frac{1}{C},&\text{if }\gamma>\alpha.
        \end{cases}
    \end{align*}
    By this choice, we can guarantee that $\eps^\alpha\leq\eta_{i,\eps}\leq\eps$ for $\eps$ small enough. Moreover, choosing $M=\max\{C,1\}$ (recall that $M$ is the required lower bound on $\mu_0$ if $\gamma=\alpha$), we can drop the second term on the left in equation \eqref{eq:energy estimate critical case} in all cases. We can then argue as in the proof above to get an estimate on $\norm{u_\eps-u}$. Finally, we note that $2\alpha-3-\gamma+\beta\geq 2\alpha-3+\gamma-\beta-3\la$ and $2\beta\geq2\gamma$ for our choice of $\beta$ and that $2\alpha-3-3\la\geq 2\alpha+\gamma-9/2$ for $\la<1/2$.
    \end{proof}
	\begin{proof}[Proof of Proposition \ref{prop:critical case gronwall estimate} and \ref{prop:subcritical case gronwall estimate}]
		We again follow \cite[Proof of Proposition 3.1]{Richard.23} and will focus on the critical case, since the subcritical case works almost the same.\\
        \textit{Step 0: Consider only configurations in $\mathcal{A}^N_{1,\alpha}\cap\mathcal{B}^N_{\la}$.} 
        We note that by construction, $v_\eps=0$ for configurations outside of $\mathcal{A}^N_{1,\alpha}\cap\mathcal{B}^N_{\la}$.
        Therefore, we will for the following steps always assume that we have a realization in $\mathcal{A}_{1,\alpha}^N\cap\mathcal{B}_{\la}^N$.
        
		\textit{Step 1: PDE for $\check{u}_\eps=w^\eps u-B_\eps(u)$.} We have that $\check{u}_\eps=0$ on $(0,T)\times\partial\Omega_\eps$, $\divg(\check{u}_\eps)=0$ and
		\begin{align*}
			\del_t \check{u}_\eps-\eps^{\gamma}\Delta \check{u}_\eps+w^\eps(u\cdot \nabla)u=w^\eps f+\widetilde{F}_\eps,
		\end{align*}
		where
		\begin{align*}
			\widetilde{F}_\eps=-w^\eps\nabla p+(M_\eps-\rho w^\eps\mathcal{R})u-\eps^\gamma\nabla q^\eps u-2\eps^\gamma\nabla w^\eps\nabla u -\eps^\gamma w^\eps\Delta u+B_\eps(\del_t u)+\eps^\gamma\Delta B_\eps(u).
		\end{align*}
        
		\textit{Step 2: Relative energy inequality for $v_\eps=\check{u}_\eps-u_\eps$.} We estimate using the energy inequality \eqref{eq:energy ineq} for $u_\eps$
		\begin{align*}
			&\frac{1}{2}\norm{v_\eps(t)}^2_{L^2(\Omega_\eps)}=\frac{1}{2}\norm{u_\eps(t)}_{L^2(\Omega_\eps)}^2+\frac{1}{2}\norm{\check{u}_\eps(t)}_{L^2(\Omega_\eps)}^2-(u_\eps(t),\check{u}_\eps(t)) \\
            &\leq\frac{1}{2}\norm{v_\eps(0)}^2_{L^2(\Omega_\eps)}-\eps^\gamma\int_0^t\norm{\nabla u_\eps}^2_{L^2(\Omega_\eps)}\ds+\int_0^t\int_{\Omega_\eps}f_\eps\cdot u_\eps\dx\ds\\
			&-\int_0^t\int_{\Omega_\eps}\del_t \check{u}_\eps\cdot u_\eps+\del_t u_\eps\cdot \check{u}_\eps\dx\ds+\int_0^t\int_{\Omega_\eps}\del_t\check{u}_\eps\cdot \check{u}_\eps\dx\ds.
		\end{align*}
		Then we can use the equations solved by $u_\eps$ with $\check{u}_\eps$ as a (divergence-free) test function:
        \begin{align}
            -\int_0^t\int_{\Omega_\eps}\del_t u_\eps\cdot \check{u}_\eps\dx\ds=\int_0^t\int_{\Omega_\eps} [(u_\eps\cdot\nabla)u_\eps]\cdot\check{u}_\eps+\eps^\gamma\nabla u_\eps:\nabla \check{u}_\eps-f_\eps\cdot\check{u}_\eps\dx\ds
        \end{align}
        In the same way, we use the equation for $\check{u}_\eps$ with $v_\eps$ as a test function to deduce
        \begin{align}
            \int_0^t\int_{\Omega_\eps}\del_t \check{u}_\eps\cdot v_\eps\dx\ds=-\int_0^t\int_{\Omega_\eps} \eps^\gamma\nabla\check{u}_\eps:\nabla v_\eps-[w^\eps(u\cdot\nabla)u]\cdot v_\eps-(w^\eps f)\cdot v_\eps\dx\ds+\langle\widetilde{F}_\eps,v_\eps\rangle.
        \end{align}
        Therefore, denoting
		\begin{align*}
			F_\eps=\widetilde{F}_\eps+w^\eps f-f_\eps,
		\end{align*}
		we have that
		\begin{align*}
			\frac{1}{2}\norm{v_\eps(t)}^2_{L^2(\Omega_\eps)}+\eps^\gamma\int_0^t \norm{\nabla v_\eps}^2_{L^2(\Omega_\eps)}\ds\leq \frac{1}{2}\norm{v_\eps(0)}^2_{L^2(\Omega_\eps)}+|I_1|+|I_2|,
		\end{align*}
		where
		\begin{align*}
			I_1&=\int_0^t\int_{\Omega_\eps}\left[(u_\eps\cdot\nabla)u_\eps\right]\cdot \check{u}_\eps-\left[w^\eps(u\cdot\nabla)u\right]\cdot v_\eps \dx\ds,\\
			I_2&=\langle F_\eps, v_\eps\rangle.
		\end{align*}
        
		\textit{Step 3: Bounding $I_1$.} Using integration by parts on the first term of $I_1$ (as well as $u_{\eps}|_{\partial \Omega_\eps}=v_{\eps}|_{\partial \Omega_\eps}=0$ and $\divg(u)=\divg(u_\eps)=0$), we can rewrite
		\begin{align*}
			I_1=&-\int_0^t\int_{\Omega_\eps}\left[(v_\eps\cdot \nabla)\check{u}_\eps\right]\cdot v_\eps+\int_0^t\int_{\Omega_\eps}\left[(I-w^\eps)(u\cdot\nabla)u\right]\cdot v_\eps\dx\ds\\
			&+\int_0^t\int_{\Omega_\eps}\left\{\left[(\check{u}_\eps-u)\cdot\nabla\right]u\right\}\cdot v_\eps\dx\ds+\int_0^t\int_{\Omega_\eps}\left[(\check{u}_\eps\cdot\nabla)(\check{u}_\eps-u)\right]\cdot v_\eps\dx\ds=I_1^1+I^2_1+I^3_1+I^4_1.
		\end{align*}
		Recalling that $\check{u}_\eps=w^\eps u-B_\eps(u)$, $\divg(v_\eps)=0$ and using integration by parts, we can rewrite
	\begin{align*}
		I_1^1=\int_0^t\int_{\Omega_\eps} \left[(v_\eps\cdot\nabla)(w^\eps u)\right]\cdot v_\eps+\left[(v_\eps\cdot\nabla)v_\eps\right]\cdot B_\eps(u)\dx\ds.
	\end{align*}
	Therefore we can estimate $I_1^1$ using the regularity assumed on $u$, \eqref{eq:corrector w1,infty estimate}, \eqref{eq:corrector gradient L1 estimate Omega_eps}, \eqref{eq:estimate Bogovskii}, \eqref{eq:corrector-Id L3 estimate} and the Sobolev embedding
	\begin{align*}
		|I_1^1|\leq& \norm{v_\eps}^2_{L^2((0,t);L^2(\Omega_\eps))}\norm{w^\eps}_{L^\infty(\Omega_\eps)}\norm{\nabla u}_{L^\infty((0,T);L^\infty(\R^3))}+\norm{\nabla w^\eps|v_\eps|^2}_{L^1((0,t);L^1(\Omega_\eps))}\norm{u}_{L^\infty((0,T);L^\infty(\R^3))}\\
		&+\norm{\nabla v_\eps}_{L^2((0,t);L^2(\Omega_\eps))}\norm{v_\eps}_{L^2((0,t);L^6(\Omega_\eps))}\norm{B_\eps(u)}_{L^\infty((0,T)L^3(\R^3))}\\
		\leqc&\norm{v_\eps}^2_{L^2(L^2)}+\sum_i \eta_{i,\eps} \norm{\nabla v_\eps}^2_{L^2(L^2(B_{\eta_{i,\eps}/2}(x_i^\eps)))}\\
        &+\norm{\nabla v_\eps}^2_{L^2(L^2)}\cdot\left(\sup_t\sum_i \eta_{i,\eps}^3\eps^{3\alpha}|\log(\eps)|\norm{\nabla u}_{L^\infty(B_{\eta_{i,\eps}/2}(x_i^\eps))}^3\right)^{1/3}.
	\end{align*}
	Next, we have the integral
	\begin{align*}
		I_1^2=\int_0^t\int_{\Omega_\eps}\left[(\id-w^\eps)(u\cdot\nabla)u\right]\cdot v_\eps\dx\ds,
	\end{align*}
	which we can estimate using \eqref{eq:corrector-Id Lp estimate p<3} and the regularity of $u$ by
	\begin{align*}
		|I_1^2|\leqc \norm{v_\eps}_{L^2((0,t);L^2(\Omega_\eps))}^2+\sum_{i=1}^N \eta_{i,\eps}\eps^{2\alpha}\norm{(u\cdot\nabla)u}_{L^2(L^\infty(\betai))}^2.
	\end{align*}
	The third integral,
	\begin{align*}
		I_1^3=\int_0^t\int_{\Omega_\eps}\left\{\left[(\check{u}_\eps-u)\cdot\nabla\right]u\right\}\cdot v_\eps\dx\ds
	\end{align*}
	can be estimated similarly, using additionally \eqref{eq:estimate Bogovskii}, so that
	\begin{align*}
		|I_1^3|&\leqc \norm{v_\eps}_{L^2((0,t);L^2(\Omega_\eps))}^2+\sum_{i=1}^N \eta_{i,\eps}\eps^{2\alpha}\norm{(u\cdot\nabla)u}_{L^2((0,T);L^\infty(\betai))}^2\\
        &\quad +\sum_i \eta_{i,\eps}^3\eps^{2\alpha}\norm{\nabla u}_{L^2((0,T);L^\infty(\betai))}^4.
	\end{align*}
	Finally, we can rewrite $I_1^4$, using integration by parts and $\divg(\check{u}_\eps)=0$, as
	\begin{align*}
		I_1^4=-\int_0^t\int_{\Omega_\eps}\left\{\left[(\check u_\eps-u)\cdot\nabla\right]v_\eps\right\}\cdot \check u_\eps\dx\ds,
	\end{align*}
	which then can be bounded by (using the regularity of $u$, \eqref{eq:corrector w1,infty estimate}, \eqref{eq:corrector-Id Lp estimate p<3}, \eqref{eq:estimate Bogovskii} and \eqref{eq:estimate Bogovskii gradient}, as well as the interpolation inequality for Sobolev spaces combined with the Sobolev embedding)
	\begin{align*}
		|I_1^4|&\leq  \frac{1}{4}\eps^\gamma\norm{\nabla v_\eps}_{L^2((0,t);L^2(\Omega_\eps))}^2+C\eps^{-\gamma}\norm{\check{u}_\eps|\check{u}_\eps-u|}_{L^2((0,t);L^2(\Omega_\eps))}^2\\
        &\leq \frac{1}{4}\eps^\gamma\norm{\nabla v_\eps}_{L^2((0,t);L^2(\Omega_\eps))}^2\\
        &\quad +C\eps^{-\gamma}(\norm{(w^\eps-\id)u}_{L^2((0,t);L^2(\R^3))}^2+\norm{B_\eps(u)}_{L^2((0,t);L^2(\R^3))}^2+\norm{ B_\eps(u)}_{L^2((0,t);L^4(\R^3))}^2)\\
		&\leq \frac{1}{4}\eps^\gamma\norm{\nabla v_\eps}_{L^2((0,t);L^2(\Omega_\eps))}^2\\
        &\quad +C\eps^{-\gamma}(\norm{(w^\eps-\id)u}_{L^2((0,t);L^2(\R^3))}^2+\norm{B_\eps(u)}_{L^2((0,t);L^2(\R^3))}^2+\norm{\nabla B_\eps(u)}_{L^2((0,t);L^2(\R^3))}^2)\\
		&\leq\frac{1}{4}\eps^\gamma\norm{\nabla v_\eps}_{L^2((0,t);L^2(\Omega_\eps))}^2+C\eps^{-\gamma}(\sum_i \eta_{i,\eps} \eps^{2\alpha}\norm{u}_{L^2(L^\infty(\betai))}^2).
	\end{align*}
    
	\textit{Step 4: Bounding $I_2$.} Now we turn our attention to $I_2$. We split $I_2$ as follows.
	\begin{align*}
		I_2=I_2^1+I_2^2+I_2^3+I_2^4,
	\end{align*}
	with
	\begin{align*}
		I_2^1&=\int_0^t\int_{\Omega_\eps}\left[(\id-w^\eps)(\nabla p-f)+f-f_\eps\right]\cdot v_\eps\dx\ds,\\
		I_2^2&=\int_0^t\int_{\Omega_\eps}\left[(w^\eps-\id)\rho\mathcal{R}u\right]\cdot v_\eps\dx\ds+\langle(M_\eps-\rho\mathcal{R})u,v_\eps\rangle,\\
		I_2^3&=-\eps^{\gamma}\int_0^t\int_{\Omega_\eps}(2\nabla w^\eps\cdot \nabla u+w^\eps\Delta u+\nabla q^\eps u)\cdot v_\eps\dx\ds,\\
		I_2^4&=\int_0^t\int_{\Omega_\eps}B_\eps(\del_t u)\cdot v_\eps+\eps^\gamma\nabla B_\eps(u):\nabla v_\eps\dx\ds.
	\end{align*}
	For the first integral $I_2^1$,
	we have by using \eqref{eq:corrector-Id Lp estimate p<3}
	\begin{align*}
		|I_2^1|&\leqc \norm{v_\eps}^2_{L^2((0,t);L^2(\Omega_\eps))}+\norm{f-f_\eps}_{L^2((0,t);L^2(\R^3))}^2+\norm{(w^\eps-\id)(\nabla p-f)}_{L^2((0,t);L^2(\R^3))}^2\\
		&\leqc \norm{v_\eps}^2_{L^2((0,t);L^2(\Omega_\eps))}+\norm{f-f_\eps}_{L^2((0,t);L^2(\R^3))}^2\\
        &\quad +\sum_i \eta_{i,\eps}\eps^{2\alpha}(\norm{\nabla p}^2_{L^2(L^\infty(\betai))}+\norm{f}^2_{L^2(L^\infty(\betai))}).
	\end{align*}
	Next, we have the integral whose estimate changes the most compared to \cite{Richard.23},
	\begin{align*}
		I_2^2=\int_0^t\int_{\Omega_\eps}\left[(w^\eps-\id)\rho\mathcal{R}u\right]\cdot v_\eps\dx\ds+\langle(M_\eps-\rho\mathcal{R})u,v_\eps\rangle.
	\end{align*}
	The first term can be estimated in the same way as for $I_2^1$, noting that $\rho\in L^\infty(\R^3)$. The second term uses Lemma \ref{lemma:right-hand side of stokes for corrector}, so that
	\begin{align*}
		|I_2^2|&\leq C_\delta\norm{v_\eps}^2_{L^2((0,t);L^2(\Omega_\eps))}+C\sum_i \eta_i\eps^{2\alpha}\norm{u}_{L^2(L^\infty(\betai))}^2\\
        &\quad +C_\delta\eps^{-\gamma}(W_2(\rho_\eps,\rho)+\eps^{1-\la})^2\norm{u}_{L^2(H^3(\R^3))}^2\\
		&\quad +C\delta\eps^\gamma\norm{\nabla v_\eps}^2_{L^2((0,t);L^2(\Omega_\eps))}+C_\delta\eps^{-3\la}\norm{u}_{L^2((0,T);L^\infty(\R^3))}^2\eps^3\sum_i (\eta_i^{-1}\eps^{3-\gamma}+\eta_i^{-2}\eps^{2\alpha})\\
		&\quad +C\eps^{3\la}\sum_i (\norm{v_\eps}_{L^2((0,t);L^2(Q_i^\eps))}^2+\delta\eps^\gamma\norm{\nabla v_\eps}_{L^2((0,t);L^2(\widetilde{Q_i^\eps}))}^2).
	\end{align*}
	For the next integral,
	\begin{align*}
		I_2^3=-\eps^{\gamma}\int_0^t\int_{\Omega_\eps}(2\nabla w^\eps\cdot\nabla u+w^\eps\Delta u)\cdot v_\eps\dx\ds+\langle\nabla q^\eps, u v_\eps\rangle,
	\end{align*}
	we use \eqref{eq:corrector-Id Lp estimate p<3}, \eqref{eq:corrector gradient L1 estimate ganzraum}, \eqref{eq:corrector gradient L1 estimate Omega_eps} to estimate
	\begin{align*}
		|I_2^3|&\leqc \eps^\gamma\int_0^t\norm{(|\nabla w^\eps|^{1/2}+|q^\eps|^{1/2})\nabla u}_{L^2(\R^3)}\norm{(|\nabla w^\eps|^{1/2}+|q^\eps|^{1/2}) v_\eps}_{L^2(\R^3)}+\norm{w^\eps}_{L^\infty(\R^3)}\norm{v_\eps}_{L^2(\Omega_\eps)}\ds\\
		&\leqc \eps^{2\gamma}+\norm{v_\eps}_{L^2((0,t);L^2(\Omega_\eps))}^2\\
        &\quad +\eps^{\gamma}\left(\sum_{i=1}^N \eta_{i,\eps}\eps^{\alpha}\norm{\nabla u}_{L^2((0,T);L^\infty(\R^3))}^2\right)^{1/2}\left(\sum_{i=1}^N\eta_{i,\eps}\norm{\nabla v_\eps}_{L^2((0,t);L^2(\betai))}^2\right)^{1/2}.
	\end{align*}
	Finally, we have
	\begin{align*}
		I_2^4=\int_0^t\int_{\Omega_\eps}B_\eps(\del_t u)\cdot v_\eps+\eps^\gamma\nabla B_\eps(u):\nabla v_\eps\dx\ds,
	\end{align*}
	which by \eqref{eq:corrector-Id Lp estimate p<3}, \eqref{eq:estimate Bogovskii} and \eqref{eq:estimate Bogovskii gradient} can be estimated by
	\begin{align*}
		|I_2^4|&\leq C\norm{v_\eps}^2_{L^2((0,t);L^2(\Omega_\eps))}+\frac{1}{4}\eps^\gamma\norm{\nabla v_\eps}_{L^2((0,t);L^2(\Omega_\eps))}^2\\
        &+C\sum_{i=1}^N \eta_{i,\eps}^3\eps^{2\alpha}\norm{\del_t u}_{L^2((0,T);L^\infty(\R^3))}^2+\eta_{i,\eps}\eps^{2\alpha+\gamma}\norm{\nabla u}_{L^2((0,T);L^\infty(\R^3))}^2.
	\end{align*}
	Now using the regularity of $u$, $p$ and $f$ (so that we can estimate their or their derivatives' $L^\infty$ norms using the Sobolev embedding), taking the expectation and applying Lemma \ref{lemma:putting together estimates with eta_i} along with recalling that $N=\eps^{-3}$, we can estimate these integrals as follows: 
    \begin{align}
        \mathbb{E}\left[I_1^1\right]&\leqc \eps^{\beta}(1+\eps^{\alpha-1}|\log(\eps)|^{1/3})\mathbb{E}\left[\norm{\nabla v_\eps}^2_{L^2((0,t);L^2(\Omega_\eps))}\right]+\mathbb{E}\left[\norm{v_\eps}_{L^2((0,t);L^2(\Omega_\eps))}^2\right],\\
        \mathbb{E}\left[I_1^2\right]&\leqc \mathbb{E}\left[\norm{v_\eps}_{L^2((0,t);L^2(\Omega_\eps))}^2\right]+\eps^{2\alpha-3+\beta},\\
        \mathbb{E}\left[I_1^3\right]&\leqc \mathbb{E}\left[\norm{v_\eps}_{L^2((0,t);L^2(\Omega_\eps))}^2\right]+\eps^{2\alpha-3+\beta},\\
        \mathbb{E}\left[I_1^4\right]&\leq \frac{1}{4}\eps^{\gamma}\mathbb{E}\left[\norm{\nabla v_\eps}^2_{L^2((0,t);L^2(\Omega_\eps))}\right]+C\eps^{2\alpha-3+\beta-\gamma},\\
        \mathbb{E}\left[I_2^1\right]&\leqc\mathbb{E}\left[\norm{v_\eps}_{L^2((0,t);L^2(\Omega_\eps))}^2+\norm{f-f_\eps}_{L^2((0,T);L^2(\R^3))}^2\right]+\eps^{2\alpha-3+\beta},\\
        \mathbb{E}\left[I_2^2\right]&\leq C_\delta\mathbb{E}\left[\norm{v_\eps}_{L^2((0,t);L^2(\Omega_\eps))}^2\right]+C\eps^{2\alpha-3+\beta}+C\delta\eps^\gamma\mathbb{E}\left[\norm{\nabla v_\eps}^2_{L^2((0,t);L^2(\Omega_\eps))}\right]\\
        &\quad\quad+C_\delta(\eps^{3-\gamma-\beta-3\la}+\eps^{-\gamma}(\mathbb{E}\left[W_2(\rho_\eps,\rho)\right]+\eps^{1-\la})^2),\\
        \mathbb{E}\left[I_2^3\right]&\leq C\mathbb{E}\left[\norm{v_\eps}_{L^2((0,t);L^2(\Omega_\eps))}^2\right]+C\delta\eps^{\gamma}\mathbb{E}\left[\norm{\nabla v_\eps}^2_{L^2((0,t);L^2(\Omega_\eps))}\right]+C_\delta \eps^{2\gamma}+\eps^{\gamma+2\beta+\alpha-3},\\
        \mathbb{E}\left[I_2^4\right]&\leq C\mathbb{E}\left[\norm{v_\eps}_{L^2((0,t);L^2(\Omega_\eps))}^2\right]+\frac{1}{4}\eps^{\gamma}\mathbb{E}\left[\norm{\nabla v_\eps}^2_{L^2((0,t);L^2(\Omega_\eps))}\right]+C(\eps^{2\alpha-3+3\beta}+\eps^{2\alpha-3+\gamma+\beta}).
    \end{align}
    Since $\beta$ and $\gamma=3-\alpha$ are positive, as well as $\eta_{i,\eps}\geq\eps^\alpha$, we observe that from all terms involving only powers of $\eps$, the largest ones are with the powers $2\alpha-3-\gamma+\beta$, $2\gamma$, $2\beta$ and $\alpha-\beta-3\la$.
    
    Then putting all these estimates together, we get for $\delta$ small enough
	\begin{align*}
		\mathbb{E}&\left[\norm{v_\eps(t)}_{L^2(\Omega_\eps)}^2 +(\eps^\gamma-C\eps^{\beta})\norm{\nabla v_\eps}_{L^2((0,t);L^2(\Omega_\eps))}^2\right]\\
		\leqc&\mathbb{E}\left[\norm{v_\eps(0)}_{L^2(\Omega_\eps)}^2+\norm{f_\eps-f}_{L^2((0,t);L^2(\R^3))}^2+C\norm{v_\eps}_{L^2((0,t);L^2(\Omega_\eps))}^2+\eps^{-\gamma}(W_2(\rho_\eps,\rho)+\eps^{1-\la})^2\right]\\
		&+(\eps^{2\alpha-3-\gamma+\beta}+\eps^{\alpha-\beta-3\la}+\eps^{2\gamma}+\eps^{2\beta}).
	\end{align*}
    Using Grönwall's lemma, we finish the proof of Proposition \ref{prop:critical case gronwall estimate}.
    
    \textit{Step 5: Subcritical case.} Finally, we consider the subcritical case. Almost all of the proof still works the same: in the PDE for $\check{u}_\eps$, all occurrences of $\eps^\gamma$ are now replaced by $\mu_0\eps^\gamma$ and since $u$ now solves the Euler equation \eqref{eq: pde for u subcrit}, the term $(M_\eps-w^\eps\rho\mathcal{R})u$ is now replaced by $\mu_0\eps^{\alpha+\gamma-3}M_\eps u$. Therefore, apart from some constants now depending on $\mu_0$, the only change of the previous integrals lies in the estimate of $I_2^2$, which now is
    \begin{align*}
        I_2^2&=\mu_0 \eps^{\alpha+\gamma-3}\int_0^t\int_{\Omega_\eps}(M_\eps u)\cdot v_\eps \dx\ds\\
        &=\mu_0 \eps^{\alpha+\gamma-3}\int_0^t\int_{\Omega_\eps}[(M_\eps-\mathcal{R}) u]\cdot v_\eps \dx\ds+\mu_0 \eps^{\alpha+\gamma-3}\int_0^t\int_{\Omega_\eps}(\mathcal{R} u)\cdot v_\eps \dx\ds.
    \end{align*}
    Like before, we can estimate this integral using Lemma \ref{lemma:right-hand side of stokes for corrector}:
    	\begin{align*}
		|I_2^2|\leq& C_\delta\norm{v_\eps}^2_{L^2((0,t);L^2(\Omega_\eps))}+C_\delta\eps^{2\alpha+2\gamma-6}\left(1+\eps^{-\gamma}(W_2(\rho_\eps,\rho)+\eps^{1-\la})^2\right)\norm{u}_{L^2((0,T);L^\infty(\R^3))}^2\\
		&+C\delta\eps^\gamma\norm{\nabla v_\eps}^2_{L^2((0,t);L^2(\Omega_\eps))}+C_\delta\eps^{2\alpha+2\gamma-6-3\la}\norm{u}^2_{L^2((0,T);L^\infty(\R^3))}\eps^3\sum_i (\eta_i^{-1}\eps^{3-\gamma}+\eta_i^{-2}\eps^{2\alpha})\\
		&+C\eps^{3\la}\sum_i (\norm{v_\eps}_{L^2((0,t);L^2(Q_i))}^2+\delta\eps^\gamma\norm{\nabla v_\eps}_{L^2((0,t);L^2(\widetilde{Q_i}))}^2).
	\end{align*}
    Putting together these estimates (note that now there is also a factor of $m_\eta$ in part of the definition of $\eta_{i,\eps}$) as before concludes the proof of Proposition \ref{prop:subcritical case gronwall estimate}.
\end{proof}
	\printbibliography

\end{document}